\documentclass{amsart}
\usepackage{amsmath}
\usepackage{amssymb}
\usepackage{amsthm}
\usepackage[all]{xy}
\usepackage{longtable}
\usepackage{wrapfig}
\usepackage[dvipdfm,bookmarks=true,bookmarksnumbered=true,bookmarkstype=toc=true,pdftitle={\@title},pdfauthor={\@author},pdfstartview={FitBH -32768}]{hyperref}
\usepackage[dvips]{graphicx}
\title{Hecke-Clifford superalgebras and crystals of type $D^{(2)}_{l}$}

\author{Shunsuke Tsuchioka}
\address{Research Institute for Mathematical Sciences, %
         Kyoto University, Kyoto 606-8502. Japan.}
\email{tshun@kurims.kyoto-u.ac.jp}

\date{}

\def\hsymbu#1{\smash{\lower1.7ex\hbox{\Huge$#1$}}}

\def\longhookrightarrow{\lhook\joinrel\longrightarrow}
\def\longtwoheadrightarrow{\relbar\joinrel\twoheadrightarrow}

\newtheorem{Thm}{Theorem}[section]
\newtheorem{Def}[Thm]{Definition}

\newtheorem{Prop}[Thm]{Proposition}
\newtheorem{Lem}[Thm]{Lemma}

\newtheorem{Cor}[Thm]{Corollary}

\newtheorem{Ex}[Thm]{Example}

\newcommand{\GEE}{\mathfrak{g}}
\newcommand{\EF}{R}
\newcommand{\MARU}{\circledast}

\newcommand{\EISO}{\simeq}
\newcommand{\ISO}{\cong}

\newcommand{\TQ}{\mathsf{Q}}
\newcommand{\IW}{\mathsf{IW}}
\newcommand{\TM}{\mathsf{M}}

\newcommand{\Te}{\widetilde{e}}
\newcommand{\Tf}{\widetilde{f}}

\newcommand{\TRIVREP}{{\boldsymbol{1}}}
\newcommand{\BII}{{\boldsymbol{i}}}
\newcommand{\BJJ}{{\boldsymbol{j}}}
\newcommand{\BKK}{{\boldsymbol{k}}}

\newcommand{\B}{\mathbb{B}}

\newcommand{\C}{\mathbb{C}}
\newcommand{\Z}{\mathbb{Z}}
\newcommand{\Q}{\mathbb{Q}}

\newcommand{\MH}{\mathcal{H}}
\newcommand{\MHI}{\mathcal{H}^{\IW}}
\newcommand{\MC}{\mathcal{C}}
\newcommand{\MA}{\mathcal{A}}
\newcommand{\FF}{{\Z/2\Z}}

\newcommand{\SYM}[1]{\mathfrak{S}_{#1}}
\newcommand{\BARR}[1]{\overline{#1}}

\newcommand{\SMOD}[1]{{#1\text{-}\mathsf{smod}}}

\newcommand{\ESMOD}[1]{{#1\text{-}\mathsf{smod}_{\BARR{0}}}}

\newcommand{\ISOM}{\stackrel{\sim}{\longrightarrow}}
\newcommand{\SISOM}{\stackrel{\sim}{\rightarrow}}
\newcommand{\COISOM}{\stackrel{\sim}{\longleftarrow}}

\newcommand{\DEF}{\stackrel{\text{def}}{=}}

\DeclareMathOperator{\INFL}{\mathsf{infl}}
\DeclareMathOperator{\SWAP}{\mathsf{swap}}

\DeclareMathOperator{\ID}{\mathsf{id}}
\DeclareMathOperator{\AD}{\mathsf{ad}}
\DeclareMathOperator{\TYPE}{\mathsf{type}}

\DeclareMathOperator{\PROJ}{\mathsf{Proj}}
\DeclareMathOperator{\SPAN}{\mathsf{span}}
\DeclareMathOperator{\KKK}{\mathsf{K}_0}

\DeclareMathOperator{\PR}{\mathsf{pr}}
\DeclareMathOperator{\WT}{\mathsf{wt}}
\DeclareMathOperator{\IRR}{\mathsf{Irr}}
\DeclareMathOperator{\HOM}{Hom}
\DeclareMathOperator{\OHOM}{\overline{\HOM}}

\DeclareMathOperator{\END}{End}

\DeclareMathOperator{\IM}{\mathsf{Im}}
\DeclareMathOperator{\CHAR}{char}
\DeclareMathOperator{\CH}{\mathsf{ch}}

\DeclareMathOperator{\SOC}{\mathsf{Soc}}
\DeclareMathOperator{\HEAD}{\mathsf{Cosoc}}
\DeclareMathOperator{\RES}{\mathsf{Res}}
\DeclareMathOperator{\REP}{\mathsf{Rep}}
\DeclareMathOperator{\PRO}{\mathsf{Proj}}

\DeclareMathOperator{\IND}{\mathsf{Ind}}
\DeclareMathOperator{\DEG}{\mathsf{deg}}

\begin{document}
\maketitle

\begin{center}
\textit{Dedicated to Professor Tetsuji Miwa on the occasion of his sixtieth birthday}
\end{center}

\begin{abstract}
In ~\cite{BK}, Brundan and Kleshchev showed that some parts of the representation theory of the
affine Hecke-Clifford superalgebras and its finite-dimensional ``cyclotomic'' quotients 
are controlled by the Lie theory of type $A^{(2)}_{2l}$ 
when the quantum parameter $q$ is a primitive $(2l+1)$-th root of
unity. We show in this paper that
similar theorems hold when $q$ is a primitive $4l$-th root of unity 
by replacing the Lie theory of type $A^{(2)}_{2l}$ with that of type $D^{(2)}_{l}$.
\end{abstract}


\section{Introduction}
It is known that we can sometimes describe the representation theory of
``Hecke algebra'' by ``Lie theory''. In this paper, we use the terminology ``Lie theory'' as
a general term for objects related to or arising from Lie algebra, such as highest weight representations,
quantum groups, Kashiwara's crystals, etc.

A famous example is Lascoux-Leclerc-Thibon's interpretation~\cite{LLT} of Kleshchev's modular 
branching rule~\cite{Kl1}.
It asserts that the modular branching graph of the symmetric groups in characteristic $p$ coincides
with Kashiwara's crystal associated with the level 1 integrable highest weight representation of the 
quantum group $U_v(\GEE(A^{(1)}_{p-1}))$. 
Brundan's modular branching rule for the Iwahori-Hecke algebras of type A at the quantum parameter $q=\sqrt[l]{1}$ over $\C$
is a similar result and can be regarded as a $q$-analogue of the above example~\cite{Br1}.

Another beautiful example is Ariki's theorem~\cite{Ari} generalizing Lascoux-Leclerc-Thibon's conjecture for
the Iwahori-Hecke algebras of type A~\cite{LLT}. It relates the decomposition numbers of the Ariki-Koike algebras
at $q=\sqrt[l]{1}$ over $\C$
and Kashiwara-Lusztig's canonical basis of a suitable integrable highest weight representation of $U_v(\GEE(A^{(1)}_{l-1}))$.
Varagnolo-Vasserot's generalization of Ariki's theorem to $q$-Schur algebras~\cite{VV}
and Yvonne's conjectural generalization for cyclotomic $q$-Schur algebras~\cite{Yvo} are also examples of 
connections between Hecke algebras and Lie theory. 

However, all the Lie theory involved so far is
only that of type $A^{(1)}_n$.
Subsequently, based on the work of Grojnowski~\cite{Gr1} and Grojnowski-Vazirani~\cite{GV}, 
Brundan and Kleshchev showed that some parts of the representation theory of the
affine Hecke-Clifford superalgebras introduced by Jones and Nazarov~\cite{JN}
and its finite-dimensional ``cyclotomic'' quotients\footnote{As a special case they include the Hecke-Clifford superalgebras introduced by Olshanski~\cite{Ols}.} introduced by Brundan and Kleshchev~\cite[\S3,\S4-b]{BK}
are controlled by the Lie theory of type $A^{(2)}_{2l}$ 
when the quantum parameter $q$ is a primitive $(2l+1)$-th root of unity.
Let $\MH_n$ be the affine Hecke-Clifford superalgebra (see Definition \ref{def_aff}) 
over an algebraically closed field $F$ of
characteristic different from 2 and let $q$ be a $(2l+1)$-th primitive
root of unity for $l\geq 1$.
Their main results are as follows.

\vspace*{0.2cm}

\noindent(1) The direct sum of the Grothendieck groups 
$K(\infty)=\bigoplus_{n\ge0} \KKK(\REP\MH_n)$ 
of the category $\REP\MH_n$ of integral $\MH_n$-supermodules
has a natural structure of 
a commutative graded Hopf $\Z$-algebra by induction and restriction~\cite[Theorem 7.1]{BK}
and the restricted dual $K(\infty)^*$ is isomorphic to the 
positive part of the Kostant $\Z$-form of 
the universal enveloping algebra of $\GEE(A^{(2)}_{2l})$
~\cite[Theorem 7.17]{BK}.

\vspace*{0.15cm}

\noindent(2) The disjoint union $B(\infty)=\bigsqcup_{n\geq 0} \IRR(\REP\MH_n)$ of the
isomorphism classes of irreducible integral $\MH_n$-supermodules
has a natural crystal structure in the sense of Kashiwara and it is isomorphic to
Kashiwara's crystal associated with $U^-_v(\GEE(A^{(2)}_{2l}))$~\cite[Theorem 8.10]{BK}.

\vspace*{0.15cm}

\noindent(3)  For each positive integral weight $\lambda$ of $A^{(2)}_{2l}$, 
one can define a finite-dimensional quotient superalgebra $\MH_n^{\lambda}$ of $\MH_n$, called
the cyclotomic
Hecke-Clifford superalgebra~\cite[\S3, \S4-b]{BK}.

\vspace*{0.15cm}

\noindent(4)  Consider the direct sums of the Grothendieck groups $K(\lambda)=\bigoplus_{n\geq 0} \KKK(\SMOD{\MH^{\lambda}_n})$ 
of the category of finite-dimensional $\MH^\lambda_n$-supermodules
and $K(\lambda)^*=\bigoplus_{n\ge0} \KKK(\PRO\MH^{\lambda}_n)$
of the category $\PRO\MH^{\lambda}_n$ of finite-dimensional 
projective $\MH^\lambda_n$-supermodules.
Then $K(\lambda)_\Q=\Q\otimes_\Z K(\lambda)$ is 
naturally identified\footnote{It is not proved so far but 
expected that the weight space decomposition of $K(\lambda)_\Q$ coincides with
the block decomposition of $\{\MH_n^\lambda\}_{n\geq 0}$ under this identification.
In fact, it is settled in the following analogous situation, when 
$\MH_n^\lambda$ is replaced by Ariki-Koike algebra~\cite{LM}, degenerate Ariki-Koike algebra~\cite{Br2}
and odd level cyclotomic quotient of the degenerate affine Sergeev superalgebra~\cite{Ruf} respectively.
See also ~\cite[\S2]{BK2}.}
with the integrable highest weight $U_\Q$-module of highest 
weight $\lambda$ where $U_\Q$ stands 
for the $\Q$-form of the universal enveloping algebra of $\GEE(A^{(2)}_{2l})$~\cite[Theorem 7.16.(i)]{BK}. 
Moreover, the 
Cartan map $K(\lambda)^*\to K(\lambda)$ is injective~\cite[Theorem 7.10]{BK} and
$K(\lambda)^*\subseteq K(\lambda)$ are dual lattices in $K(\lambda)_\Q$ 
under the Shapovalov form~\cite[Theorem 7.16.(iii)]{BK}.

\vspace*{0.15cm}

\noindent(5)  The disjoint union $B(\lambda)=\bigsqcup_{n\ge0} \IRR(\SMOD{\MH^{\lambda}_n})$
is isomorphic to Kashiwara's crystal associated with the integrable $U_v(\GEE(A^{(2)}_{2l}))$-module of
highest weight $\lambda$~\cite[Theorem 8.11]{BK}.

\vspace*{0.2cm}

Analogous results for the degenerate affine Sergeev superalgebras of Nazarov~\cite{Naz} and 
its cyclotomic quotients~\cite[\S4-i]{BK} over 
an algebraically closed field $F$ of $\CHAR F=2l+1$ are also established in ~\cite{BK}
parallel to those for the affine Hecke-Clifford superalgebras and its cyclotomic quotients at $q=\sqrt[2l+1]{1}$
over an algebraically closed field $F$ of $\CHAR F\ne 2$.
As a very special corollary of the results for the degenerate superalgebras,
they beautifully obtain a modular branching rule of the spin symmetric groups $\widehat{\mathfrak{S}}_n$.
However, it may be a reason why they deal only with the case $q=\sqrt[2l+1]{1}$ for 
the affine Hecke-Clifford superalgebras in ~\cite{BK}.

Note that exactly the same results as above hold 
when $q$ is a primitive $2(2l+1)$-th root of unity for $l\geq 1$.
This follows from the fact that $-q$ is a primitive $(2l+1)$-th root of unity and the superalgebra isomorphism 
between the affine Hecke-Clifford superalgebras (see Definition \ref{def_aff}) $\MH_n(q)$ and $\MH_n(-q)$ given by
\begin{align*}
\MH_n(q)\ISOM \MH_n(-q),\quad X_i\longmapsto X_i, \quad C_i\longmapsto C_i,\quad T_j\longmapsto -T_j
\end{align*}
for $1\leq i\leq n$ and $1\leq j<n$.
However, the case when the multiplicative order of $q$ is divisible by 4 is yet untouched.

The purpose of this paper is to show that Brundan-Kleshchev's method is still applicable to
the case when $q$ is a primitive $4l$-th root of unity for any $l\geq 2$. In this case we have 
very similar results by replacing $A^{(2)}_{2l}$ with $D^{(2)}_{l}$ in the above summary.
Roughly speaking, we prove the following four statements
(for the precise statements, 
see Corollary \ref{final_thm2}, Corollary \ref{final_thm1}, Theorem \ref{final_thm3} and 
Theorem \ref{final_thm4}).

\begin{Thm}
Let $F$ be an algebraically closed field of
characteristic different from 2 and let $q$ be a primitive $4l$-th 
root of unity for $l\geq 2$. For each positive integral weight $\lambda$ of $D^{(2)}_{l}$, 
we can define a finite-dimensional quotient superalgebra $\MH_n^{\lambda}$ of $\MH_n$ (see Definition \ref{cyc_quo})
so that the followings hold.
\begin{enumerate}
\item the graded dual of $K(\infty)=\bigoplus_{n\ge0} \KKK(\REP\MH_n)$ is isomorphic to $U^+_\Z$
as graded $\Z$-Hopf algebra (see Theorem \ref{final_thm4}).
\item $K(\lambda)_\Q=\bigoplus_{n\geq 0} \Q\otimes \KKK(\SMOD{\MH^{\lambda}_n})$ has a left $U_\Q$-module structure
which is isomorphic to the integrable highest weight $U_\Q$-module of highest weight $\lambda$ (see Theorem \ref{final_thm3} for details).
\item $B(\infty)=\bigsqcup_{n\geq 0} \IRR(\REP\MH_n)$ is isomorphic to
Kashiwara's crystal associated with $U^-_v(\GEE(D^{(2)}_{l}))$ (see Corollary \ref{final_thm2}).
\item $B(\lambda)=\bigsqcup_{n\ge0} \IRR(\SMOD{\MH^{\lambda}_n})$
is isomorphic to Kashiwara's crystal associated with the integrable $U_v(\GEE(D^{(2)}_{l}))$-module of
highest weight $\lambda$ (see Corollary \ref{final_thm1}).
\end{enumerate}
Here $U^+_\Z$ is the positive part of the
Kostant $\Z$-form of the universal enveloping algebra of $\GEE(D^{(2)}_{l})$ and 
$U_\Q$ is the $\Q$-subalgebra of the universal enveloping algebra of $\GEE(D^{(2)}_{l})$ generated by
the Chevalley generators (see \S\ref{lt}).
\label{int}
\end{Thm}

A difference between our paper and ~\cite{BK} is a behavior of representations of 
low rank affine Hecke-Clifford
superalgebras which are treated at length in \S\ref{keisan}. 

Finally, let us explain a reason behind our searching the ``missing''
connection between Hecke algebra and Lie theory of type $D^{(2)}_{n+1}$.
It is well known that the level 1 crystal $\B(\Lambda_0)$ associated with $U_v(A^{(1)}_n)$ or $U_v(A^{(2)}_{2n})$
is described by partitions~\cite{MM,Kan}.
It is interesting that some of the combinatorics appearing in their descriptions
had been already studied in the representation theory of the (spin) symmetric groups~\cite{Jam,Mor,MY},
and such combinatorics controls modular branching of the (spin) symmetric groups~\cite{Kl1,Kl2,BK}.
Thus, it is natural to ask which level 1 crystal has such a combinatorial realization, i.e.,
its underlying set is a subset of the set of partitions.

This problem is related to Kyoto path model~\cite{KKMMNN2,KKMMNN} or 
its combinatorial counterpart, Kang's Young wall~\cite{Kan}.
The key tool underlying their realizations is a notion of perfect crystal~\cite[Definition 1.1.1]{KKMMNN}
which is introduced in ~\cite{KKMMNN2} to compute one-point functions of 
vertex models in statistical mechanics. 
As seen in ~\cite{Kan}, in order to realize $\B(\Lambda_0)$ as a subset 
of the set of partitions, we need a perfect crystal of level 1 which 
has no branching point\footnote{Let $G=(V,E)$ be a directed graph meaning that
$V$ is the set of vertices and $E\subseteq V\times V$
is the adjacent relations meaning that $(v,w)\in E$ if and only if
there exists a directed arrow from $v$ to $w$.
We say that a vertex $w$ is a branching point of $G$
if there exist $u$ and $v$ such that
$u\ne v, u\ne w, v\ne w, (w,u)\in E$ and $(w,v)\in E$.}.
As shown in ~\cite{KKMMNN}, such a perfect crystal of level 1 exists in types $A^{(1)}_n, A^{(2)}_{2n}$ and $D^{(2)}_{n+1}$.
Conversely, we can show that a pair of affine type and its perfect crystal of level 1 which has no
branching point is one of the followings\footnote{$(A^{(1)}_1,(B^{1,1})^{\otimes 2})$ can be interpreted formally
as $n=1$ case of $(D^{(2)}_{n+1},B^{1,1})$.}
\begin{align*}
(A^{(1)}_1,B^{1,1}), \quad(A^{(1)}_1,(B^{1,1})^{\otimes 2}), \quad(A^{(1)}_n,B^{1,1}) (n\geq 2), \\
(A^{(1)}_n,B^{n,1}) (n\geq 2), \quad(A^{(2)}_{2n},B^{1,1}) (n\geq 1), \quad(D^{(2)}_{n+1},B^{1,1}) (n\geq 2)
\end{align*}
if we assume the conjecture that any perfect crystal is a finite number of tensor product of
Kirillov-Reshetikhin perfect crystals $B^{r,s}$ as stated in the first paragraph of the introduction of ~\cite{KNO}
and also assume the conjectural properties~\cite[Conjecture 2.1]{HKOTY}~\cite[Conjecture 2.1]{HKOTZ}
of Kirillov-Reshetikhin modules $W^{(r)}_{s}$.

This crystal-theoretic fact distinguishes types $A^{(1)}_n, A^{(2)}_{2n}$ and $D^{(2)}_{n+1}$ from the other
affine types and it is a reason behind our searching the ``missing''
connection between Hecke algebra and Lie theory of type $D^{(2)}_{n+1}$.

%

\vspace*{0.25cm}

\noindent\textbf{Organization of the paper}
The paper is organized as follows.
In \S\ref{Lie_theory_section}, we recall our conventions and necessary facts for superalgebras, supermodules and
Kashiwara's crystal theory. In \S\ref{affine_section} (resp.\ \S\ref{cyc_quot_section}), we 
define the affine Hecke-Clifford superalgebras (resp. the cyclotomic Hecke-Clifford superalgebras) and
review fundamental theorems for them along with ~\cite{BK}.
In \S\ref{keisan}, we give some preparatory character calculations 
concerning behavior of representations of low rank affine Hecke-Clifford superalgebras $\MH_2,\MH_3$ and $\MH_4$
which are responsible for the appearance of Lie theory of type $D^{(2)}_l$.
Finally, in \S\ref{families} we prove Theorem \ref{int}.

\vspace*{0.25cm}

\noindent\textbf{Acknowledgements}
The author would like to express his profound gratitude 
to Professor Jonathan Brundan and Professor Alexander Kleshchev
since he owes a lot to their paper ~\cite{BK}.
Actually, many parts of the arguments
in this paper are essentially the same as theirs.
He also would like to thank Professor Masaki Kashiwara and Professor Susumu Ariki for
their valuable comments when he read ~\cite{Kl2} in their seminar and Professor Masato Okado 
for kindly answering questions on perfect crystals.
Especially, he is grateful to Professor Masaki Kashiwara for reading the manuscript and giving 
him many useful comments and to Professor Jonathan Brundan for answering questions on ~\cite{BK}
and updated ~\cite{BK2}.
He is supported by JSPS fellowships for Young Scientists (No. 20-5306).

\section{Preliminaries}
\label{Lie_theory_section}
\subsection{Superalgebras and supermodules}
\label{salg}
We briefly recall our conventions and notations for superalgebras and supermodules following ~\cite[\S2-b]{BK} 
(see also the references therein). In the rest of the paper, 
we always assume that our field $F$ is algebraically
closed with $\CHAR F\ne 2$.

By a vector superspace, we mean a $\FF$-graded vector space $V=V_{\BARR{0}}\oplus V_{\BARR{1}}$ over $F$ and 
denote the parity of a homogeneous vector $v\in V$ by $\BARR{v}\in\FF$.
Given two vector superspaces $V$ and $W$, an $F$-linear map $f:V\to W$ is called homogeneous
if there exists $p\in\FF$ such that $f(V_i)\subseteq W_{p+i}$ for $i\in\FF$. 
In this case we call $p$ the parity of $f$ and denote it by $\BARR{f}$.

A superalgebra $A$ is a vector superspace which is an unital associative $F$-algebra such that 
$A_iA_j\subseteq A_{i+j}$ for $i,j\in \FF$.
By an $A$-supermodule, we mean a vector superspace $M$ 
which is a left $A$-module such that $A_iM_j\subseteq M_{i+j}$ for $i,j\in \FF$.
In the rest of the paper, we only deal with finite-dimensional $A$-supermodules.
Given two $A$-supermodules $V$ and $W$,
an $A$-homomorphism $f:V\to W$ is an 
$F$-linear map such that
\begin{align*}
f(av)=(-1)^{\BARR{f}\BARR{a}}af(v)
\end{align*}
for $a\in A$ and $v\in V$. We denote the set of $A$-homomorphisms from $V$ to $W$ by $\HOM_A(V,W)$.
By this, we can form a superadditive category $\SMOD{A}$ whose hom-set is a vector superspace
in a way that is compatible with composition. 
However, we adapt a slightly different definition of isomorphisms
from the categorical one\footnote{Note that for irreducible $A$-supermodules $V$ and $W$, the following statements 
are equivalent. 
\begin{enumerate}
\item there exist $f\in\HOM_A(V,W)$ and $g\in\HOM_A(W,V)$ such that $f\circ g=\ID_W$ and $g\circ f=\ID_V$. 
\item  there exist $f\in\HOM_A(V,W)$ and $g\in\HOM_A(W,V)$ which are both homogeneous and satisfy $f\circ g=\ID_W, g\circ f=\ID_V$.
\end{enumerate}
}.
Two $A$-supermodules $V$ and $W$ are called evenly isomorphic (and denoted by $V\EISO W$)
if there exists an even $A$-homomorphism $f:V\to W$ which is an $F$-vector space isomorphism.
They are called isomorphic (and denoted by $V\ISO W$)
if $V\EISO W$ or $V\EISO \Pi W$.
Here for an $A$-supermodule $M$, $\Pi M$ is an $A$-supermodule defined by
the same but the opposite grading underlying vector superspace
$(\Pi M)_i=M_{i+\BARR{1}}$ for $i\in\FF$ and a new action given as follows from the old one
\begin{align*}
a\cdot_{\mathsf{new}} m=(-1)^{\BARR{a}}a\cdot_{{\mathsf{old}}}m.
\end{align*}
We denote the isomorphism class of an $A$-supermodule $M$ by $[M]$ and denote the set of
isomorphism classes of irreducible $A$-supermodules by $\IRR(\SMOD{A})$.

Given two superalgebras $A$ and $B$, $A\otimes B$ with multiplication defined by
\begin{align*}
(a_1\otimes b_1)(a_2\otimes b_2)=(-1)^{\BARR{b_1}\BARR{a_2}}(a_1a_2)\otimes (b_1b_2)
\end{align*}
for $a_i\in A, b_j\in B$ is again a superalgebra.
Let $V$ be an $A$-supermodule and let $W$ be a $B$-supermodule.
Their tensor product $V\otimes W$ is an $A\otimes B$-supermodule
by the action given by
\begin{align*}
(a\otimes b)(v\otimes w)=(-1)^{\BARR{b}\BARR{v}}(av)\otimes (bw)
\end{align*}
for $a\in A, b\in B, v\in V, w\in W$.
Let us assume that $V$ and $W$ are both irreducible.
We say that $V$ is type $\TQ$ if $V\EISO\Pi V$ otherwise type $\TM$.
If $V$ and $W$ are both of type $\TQ$, then there exists a unique 
(up to odd isomorphism) irreducible $A\otimes B$-supermodule $X$ of type $\TM$
such that 
\begin{align*}
V\otimes W\EISO X\oplus\Pi X
\end{align*}
as $A\otimes B$-supermodules. We denote $X$ by $V\MARU W$.
Otherwise $V\otimes W$ is irreducible but we also write it as $V\MARU W$. 
Note that $V\MARU W$ is defined only up to isomorphism in general and
$V\MARU W$ is of type $\TM$
if and only if $V$ and $W$ are of the same type.

We extend the operation $\MARU$ as follows. 
Let $A$ and $B$ be superalgebras and let
$V$ be an $A$-supermodule and let $W$ be a $B$-supermodule.
Consider a pair $(V,\theta_V)$ where $\theta_V$ is either an odd involution of $V$ or $\theta_V=\ID_V$, and
also consider a similar pair $(W,\theta_W)$.
If $\theta_V=\ID_V$ or $\theta_W=\ID_W$, then we define $(V,\theta_V)\MARU (W,\theta_W)=V\otimes W$.
If $\theta_V$ and $\theta_W$ are both odd involutions, then
\begin{align*}
\theta_V\otimes \theta_W:V\otimes W\longrightarrow V\otimes W,\quad v\otimes w\longmapsto
(-1)^{\BARR{v}}\theta_V(v)\otimes\theta_W(w)
\end{align*}
is an even $A\otimes B$-supermodule homomorphism such that $(\theta_V\otimes \theta_W)^2=-\ID_{V\otimes W}$.
Thus, $V\otimes W$ decomposes into $\pm\sqrt{-1}$-eigenspaces $X_{\pm\sqrt{-1}}$.
Note that $X_{+\sqrt{-1}}$ and $X_{-\sqrt{-1}}$ are oddly isomorphic
since we have
\begin{align*}
(\theta_V\otimes\ID_{W})(X_{\pm\sqrt{-1}})=(\ID_V\otimes\theta_W)(X_{\pm\sqrt{-1}})=X_{\mp\sqrt{-1}}.
\end{align*}
Now we define $(V,\theta_V)\MARU (W,\theta_W)=X_{\sqrt{-1}}$.
Of course, we can pick the other summand, but such specification 
makes arguments simpler when we consider functoriality.

We also introduce $\HOM$ version of this operation.
Assume further that $B$ is a subsuperalgebra of $A$.
If $\theta_V=\ID_V$ or $\theta_W=\ID_W$, then we define $\OHOM_{B}((W,\theta_W),(V,\theta_V))=\HOM_{B}(W,V)$ which
can be regarded as a supermodule over $C(A,B)\DEF\{a\in A\mid \textrm{$ab=(-1)^{\BARR{a}\BARR{b}}ba$ for all $b\in B$}\}$ 
by the action
$(cf)(v)=c(f(v))$ for $c\in C(A,B)$ and $f\in \HOM_{B}(W,V)$.
If $\theta_V$ and $\theta_W$ are both odd involutions, then
\begin{align*}
\Theta:\HOM_{B}(W,V)\longrightarrow \HOM_{B}(W,V),\quad
f\longmapsto (\Theta(f))(v)=(-1)^{\BARR{f}}\theta_V(f(\theta_W(v)))
\end{align*}
is an even $C(A,B)$-supermodule homomorphism such that $\Theta^2=\ID_{\HOM_{B}(W,V)}$.
Thus, $\HOM_{B}(W,V)$ decomposes into $\pm 1$-eigenspaces $X_{\pm 1}$.
Similarly, we see that $X_{\pm 1}\EISO \Pi X_{\mp 1}$, and we define $\OHOM_{B}((W,\theta_W),(V,\theta_V))=X_{+1}$.

For a superalgebra $A$, we define the Grothendieck group $\KKK(\SMOD{A})$ to be
the quotient of the $\Z$-module freely generated by all finite-dimensional $A$-supermodules
by the $\Z$-submodule generated by
\begin{itemize}
\item $V_1-V_2+V_3$ for every short exact sequence $0\to V_1\to V_2\to V_3\to 0$ in $\ESMOD{A}$.
\item $M-\Pi M$ for every $A$-supermodule $M$.
\end{itemize}
Here $\ESMOD{A}$ is the abelian subcategory of $\SMOD{A}$ whose 
objects are the same but morphisms are consisting of even $A$-homomorphisms.
Clearly, $\KKK(\SMOD{A})$ is a free $\Z$-module with basis $\IRR(\SMOD{A})$.
The importance of the operation $\MARU$ lies in the fact that it gives an isomorphism
\begin{align}
\KKK(\SMOD{A})\otimes_\Z\KKK(\SMOD{B})
\ISOM
\KKK(\SMOD{A\otimes B}),\quad
[V]\otimes [W]\longmapsto [V\MARU W]
\label{K_isom}
\end{align}
for two superalgebras $A$ and $B$.


Finally, we make some remarks on projective supermodules.
Let $A$ be a superalgebra. A projective $A$-supermodule is, by definition, a projective object in $\SMOD{A}$ and
it is equivalent to saying that it is a projective object in $\ESMOD{A}$ since there are canonical isomorphisms
\begin{align*}
\HOM_{\SMOD{A}}(V,W)_{\BARR{0}}&\cong \HOM_{\ESMOD{A}}(V,W),\\
\HOM_{\SMOD{A}}(V,W)_{\BARR{1}}&\cong \HOM_{\ESMOD{A}}(V,\Pi W)(\cong \HOM_{\ESMOD{A}}(\Pi V,W)).
\end{align*}
We denote by $\PROJ A$ the full subcategory of $\SMOD{A}$ consisting of all the projective $A$-supermodules.

Let us assume further that $A$ is finite-dimensional.
Then, as in the usual finite-dimensional algebras, every $A$-supermodule $X$ has 
a (unique up to even isomorphism) projective cover $P_X$ in $\ESMOD{A}$.
If $X$ is irreducible, then it is (evenly) isomorphic to a principal indecomposable $A$-supermodule. 
From this, we easily see $M\cong N$ if and only if $P_M\cong P_N$ for $M,N\in\IRR(\SMOD{A})$.
Thus, $\KKK(\PROJ A)$ is identified with $\KKK(\SMOD{A})^*\DEF\HOM_\Z(\KKK(\SMOD{A}),\Z)$
through the non-degenerate canonical pairing
\begin{align*}
\langle,\rangle_A:\KKK(\PROJ A)\times \KKK(\SMOD{A})
&\longrightarrow
\Z,\\
([P_M],[N])
&\longmapsto 
\begin{cases}
\dim\HOM_A(P_M,N) & \textrm{if $\TYPE M=\TM$}, \\
\frac{1}{2}\dim\HOM_A(P_M,N) & \textrm{if $\TYPE M=\TQ$},
\end{cases}
\end{align*}
for all $M\in\IRR(\SMOD{A})$ and $N\in\SMOD{A}$. 
Note that the left hand side is nothing but the composition multiplicity $[N:M]$. 
We also reserve the symbol 
\begin{align*}
\omega_A:\KKK(\PROJ A)\longrightarrow \KKK(\SMOD{A})
\end{align*}
for 
the natural Cartan map.


\subsection{Lie theory}
\label{lt}

We review necessary Lie theory for our purpose.
Note that all the Lie-theoretic objects are considered over $\C$ as usual although we are considering
representations of ``Hecke superalgebra'' over $F$.

Let $A=(a_{ij})_{i,j\in I}$ be a symmetrizable generalized Cartan matrix and
let $\GEE$ be the corresponding Kac-Moody Lie algebra.
We denote the weight lattice by $P$,
the set of simple roots by $\{\alpha_i\mid i\in I\}$ and 
the set of simple coroots by $\{h_i\mid i\in I\}$, etc.\ as usual.
We denote by $U_\Q$ the $\Q$-subalgebra of the universal enveloping algebra of $\GEE$ generated by
the Chevalley generators $\{e_i,f_i,h_i\mid i\in I\}$. 
In other words, $U_\Q$ is a $\Q$-subalgebra generated by $\{e_i,f_i,h_i\mid i\in I\}$ with
the following relations 
\begin{align}
\begin{split}
[h_i,h_j]=0,\quad
[h_i,e_j]=a_{ij}e_j,\quad
[h_i,f_j]=-a_{ij}f_j,\quad \\
[e_i,f_j]=\delta_{ij}h_i.\quad
(\AD e_i)^{1-a_{ik}}(e_k)=(\AD f_i)^{1-a_{ik}}(f_k)=0,
\label{envelop}
\end{split}
\end{align}
for all $i,j,k\in I$ with $i\ne k$.
We also denote by $U^+_\Z$ (resp.\ $U^-_\Z$) the positive (resp.\ negative) part 
of the Kostant $\Z$-form of $U_\Q$, i.e., $U^+_\Z$ (resp.\ $U^-_\Z$) is a subalgebra of $U_\Q$ 
generated by
the divided powers $\{e_i^{(n)}\DEF e_i^n/n!\mid n\geq 1\}$ (resp.\ $\{f_i^{(n)}\mid n\geq 1\}$).

Next, we recall the notion of Kashiwara's crystal following ~\cite{Ka1}.
\begin{Def}
A $\GEE$-crystal is a 6-tuple $(B,\WT,\{\varepsilon_i\}_{i\in I}, \{\varphi_i\}_{i\in I},\{\Te_i\}_{i\in I},\{\Tf_i\}_{i\in I})$
\begin{align*}
\WT &: B\longrightarrow P, \\
\varepsilon_i,\varphi_i &: B\longrightarrow \Z\sqcup\{-\infty\}, \\
\Te_i,\Tf_i &: B\sqcup\{0\}\longrightarrow B\sqcup\{0\}
\end{align*}
satisfies the following axioms.
\begin{enumerate}
\item For all $i\in I$, we have $\Te_i0=\Tf_i0=0$.
\label{def_of_crystal1}
\item For all $b\in B$ and $i\in I$, we have $\varphi_i(b)=\varepsilon_i(b)+\WT(b)(h_i)$.
\label{def_of_crystal2}
\item For all $b\in B$ and $i\in I$, $\Te_ib\ne 0$ implies 
$\varepsilon_i(\Te_ib)=\varepsilon_i(b)-1,
\varphi_i(\Te_ib)=\varphi_i(b)+1$ and $\WT(\Te_ib)=\WT(b)+\alpha_i$.
\label{def_of_crystal3}
\item For all $b\in B$ and $i\in I$, $\Tf_ib\ne 0$ implies 
$\varepsilon_i(\Tf_ib)=\varepsilon_i(b)+1,\varphi_i(\Tf_ib)=\varphi_i(b)-1$ and $\WT(\Tf_ib)=\WT(b)-\alpha_i$.
\label{def_of_crystal4}
\item For all $b,b'\in B$ and $i\in I$, $b'=\Tf_ib$ is equivalent to $b=\Te_ib'$.
\label{def_of_crystal5}
\item For all $b\in B$ and $i\in I$, $\varphi_i(b)=-\infty$ implies $\Te_ib=\Tf_ib=0$.
\label{def_of_crystal6}
\end{enumerate}
\label{def_of_crystal}
\end{Def}

\begin{Def}
Let $B$ be a $\GEE$-crystal. The crystal graph associated with $B$ (and usually denoted by 
the same symbol $B$) is an $I$-colored directed graph 
whose vertices are the elements of $B$ and there is an $i$-colored directed edge from $b$ to $b'$
if and only if $b'=\Tf_ib$ for $b,b'\in B$ and $i\in I$.
\end{Def}

\begin{Def}
Let $B$ and $B'$ be $\GEE$-crystals.
Their tensor product crystal $B\otimes B'$ is a $\GEE$-crystal defined as follows.
\allowdisplaybreaks{
\begin{align*}
B\otimes B'& =B\times B', \\
\varepsilon_i(b\otimes b') &= \max(\varepsilon_i(b),\varepsilon_i(b')-\WT(b)(h_i)), \\
\varphi_i(b\otimes b') &= \max(\varphi_i(b)+\WT(b')(h_i),\varphi_i(b')), \\
\Te_i(b\otimes b') &= \begin{cases}
\Te_ib\otimes b' & \textrm{if $\varphi_i(b)\geq \varepsilon_i(b')$}, \\
b\otimes \Te_ib' & \textrm{if $\varphi_i(b)< \varepsilon_i(b')$}, \\
\end{cases} \\
\Tf_i(b\otimes b') &= \begin{cases}
\Tf_ib\otimes b' & \textrm{if $\varphi_i(b)> \varepsilon_i(b')$}, \\
b\otimes \Tf_ib' & \textrm{if $\varphi_i(b)\leq \varepsilon_i(b')$}, \\
\end{cases} \\
\WT(b\otimes b') &= \WT(b)+\WT(b').
\end{align*}}
Here we regard $b\otimes 0$ and $0\otimes b$ as $0$.
\label{cry_tensor}
\end{Def}

\begin{Def}
Let $B$ and $B'$ be $\GEE$-crystals.
A crystal morphism $g:B\to B'$ is a map $g:B\sqcup\{0\}\to B'\sqcup\{0\}$ such that
\begin{enumerate}
\item $g(0)=0$.
\item If $b\in B$ and $g(b)\in B'$, then we have $\WT(g(b))=\WT(b)$, $\varepsilon_i(g(b))=\varepsilon_i(b)$ and
$\varphi_i(g(b))=\varphi_i(b)$ for all $i\in I$.
\item For $b\in B$ and $i\in I$, we have $g(\Te_i b)=\Te_ig(b)$ if
$g(b)\in B'$ and $g(\Te_i b)\in B'$.
\item For $b\in B$ and $i\in I$, we have $g(\Tf_i b)=\Tf_ig(b)$ if
$g(b)\in B'$ and $g(\Tf_i b)\in B'$.
\end{enumerate}
If it commutes with all $\Te_i$ (resp.\ $\Tf_i$), then we call it an $e$-strict (resp.\ $f$-strict) morphism.
We call it a crystal embedding if it is injective, $e$-strict and $f$-strict.
\end{Def}

\begin{Ex}
For each $\lambda\in P^+$, we denote by $T_{\lambda}=\{t_{\lambda}\}$ the $\GEE$-crystal defined by
\begin{align*}
\WT(t_\lambda)=\lambda,\quad\varphi_i(t_\lambda)=\varepsilon_i(t_\lambda)=-\infty,\quad\Te_it_\lambda=\Tf_it_\lambda=0.
\end{align*}
\end{Ex}

\begin{Ex}
For each $i\in I$, we denote by $B_i=\{b_i(n)\mid n\in \Z\}$ the $\GEE$-crystal defined by
$\WT(b_i(n))=n\alpha_i$ and 
\begin{align*}
\varepsilon_j(b_i(n)) &= 
\begin{cases}
-n & \textrm{if $j=i$}, \\
-\infty & \textrm{if $j\ne i$},
\end{cases}\quad
\varphi_j(b_i(n)) = 
\begin{cases}
n & \textrm{if $j=i$}, \\
-\infty & \textrm{if $j\ne i$},
\end{cases}\\
\Te_j(b_i(n)) &= 
\begin{cases}
b_i(n+1) & \textrm{if $j=i$}, \\
0 & \textrm{if $j\ne i$},
\end{cases}\quad
\Tf_j(b_i(n)) = 
\begin{cases}
b_i(n-1) & \textrm{if $j=i$}, \\
0 & \textrm{if $j\ne i$}.
\end{cases}
\end{align*}
\end{Ex}

These pathological $\GEE$-crystals are utilized 
in the following characterizations~\cite[Proposition 3.2.3]{KS}~\cite[Proposition 2.3.1]{Sai}.
\begin{Prop}
We denote by $\B(\infty)$ the associated $\GEE$-crystal with the crystal base of $U_v^-(\GEE)$.
Let $B$ be a $\GEE$-crystal and $b_0$ an element of $B$ with $\WT(b_0)=0$.
If  the following conditions hold, then $B$ is isomorphic to $\B(\infty)$.
\begin{enumerate}
\item $\WT(B)\subseteq \sum_{i\in I}\Z_{\leq 0}\alpha_i$.
\item $b_0$ is a unique element of $B$ such that $\WT(b_0)=0$.
\item $\varepsilon_i(b_0)=0$ for every $i\in I$.
\item $\varphi_i(b)\in\Z$ for any $b\in B$ and $i\in I$.
\item For every $i\in I$, there exists a crystal embedding $\Psi_i:B\to B\otimes B_i$ such that $\Psi_i(B)\subseteq B\times \{\Tf_i^nb_i(0)\mid n\geq 0\}$.
\item For any $b\in B$ such that $b\ne b_0$, there exists $i\in I$ such that $\Psi_i(b)=b'\otimes \Tf_i^nb_i(0)$
with $n> 0$.
\end{enumerate}
\label{characterization_thm1}
\end{Prop}

\begin{Prop}
We denote by $\B(\lambda)$ the associated $\GEE$-crystal with the crystal base of 
the integrable highest $U_v(\GEE)$-module of highest weight $\lambda\in P^+$.
Let $B$ be a $\GEE$-crystal and $b_\lambda$ an element of $B$ with $\WT(b_\lambda)=\lambda$.
If the following conditions hold, then $B$ is isomorphic to $\B(\lambda)$.
\begin{enumerate}
\item $b_\lambda$ is a unique element of $B$ such that $\WT(b_\lambda)=\lambda$.
\item There is an $f$-strict crystal morphism $\Phi:B(\infty)\otimes T_\lambda\to B$ such that
$\Phi(b_0\otimes t_{\lambda})=b_{\lambda}$ and $\IM\Phi=B\sqcup\{0\}$. Here $b_0$ is the unique element of $B(\infty)$
with $\WT(b_0)=0$.
\item Consider the set $\{b\in B(\infty)\otimes T_\lambda\mid \Phi(b)\ne 0\}$. Then it is isomorphic to $B$ through
$\Phi$ as a set.
\item For any $b\in B$ and $i\in I$, $\varepsilon_i(b)=\max\{k\geq 0\mid \Te_i^k(b)\ne 0\}$
and $\varphi_i(b)=\max\{k\geq 0\mid \Tf_i^k(b)\ne 0\}$.
\end{enumerate}
\label{characterization_thm2}
\end{Prop}


\section{Affine Hecke-Clifford superalgebras of Jones and Nazarov}
\label{affine_section}
\subsection{Definition and vector superspace structure}

From now on, we reserve a non-zero quantum parameter $q\in F^\times$ and set $\xi=q-q^{-1}$ for convenience.
Let us define our main ingredient $\MH_n$, affine Hecke-Clifford superalgebra~\cite[\S3]{JN}.
Although Jones and Nazarov introduced it under the name of affine Sergeev algebra, we 
call it affine Hecke-Clifford superalgebra following ~\cite[\S2-d]{BK}.

\begin{Def}
Let $n\geq 0$ be an integer. The affine Hecke-Clifford superalgebra $\MH_n$ is defined 
by even generators $X_1^{\pm1}, \cdots, X_n^{\pm1}, T_1,\cdots, T_{n-1}$ and 
odd generators $C_1,\cdots,C_n$ with the following relations.
\begin{enumerate}
\item $X_{i}X_{i}^{-1}=X_{i}^{-1}X_i=1, X_iX_j=X_iX_j$ for all $1\leq i,j\leq n$.
\item $C_i^2=1, C_iC_j+C_jC_i=0$ for all $1\leq i\ne j\leq n$.
\label{clifford_rel}
\item $T_i^2=\xi T_i+1, T_iT_j=T_jT_i, T_kT_{k+1}T_k=T_{k+1}T_kT_{k+1}$ for all $1\leq k\leq n-2$ and $1\leq i,j\leq n-1$ 
with $|i-j|\geq 2$.
\label{iwahori_rel}
\item $C_iX_i^{\pm1}=X_i^{\mp1}C_i, C_iX_j^{\pm1}=X_j^{\pm1}C_i$ for all $1\leq i\ne j\leq n$.
\item $T_iC_i=C_{i+1}T_i, (T_i+\xi C_iC_{i+1})X_iT_i=X_{i+1}$ for all $1\leq i\leq n-1$.
\label{non_triv_eq}
\item $T_iC_j=C_jT_i, T_iX_j^{\pm1}=X_j^{\pm1}T_i$ for all $1\leq i\leq n-1$ and $1\leq j\leq n$ with $j\ne i, i+1$.
\end{enumerate}
\label{def_aff}
\end{Def}

Note that the relations in Definition \ref{def_aff} implies the followings for $1\leq i\leq n-1$.
\allowdisplaybreaks{
\begin{align}
T_iC_{i+1} &= C_iT_i-\xi(C_i-C_{i+1}) \label{non_triv_eq1}, \\
T_iX_{i} &= X_{i+1}T_i-\xi(X_{i+1}+C_iC_{i+1}X_i) \label{non_triv_eq2}, \\
T_iX^{-1}_{i} &= X^{-1}_{i+1}T_i+\xi(X^{-1}_{i}+X^{-1}_{i+1}C_iC_{i+1}) \label{non_triv_eq3}.
\end{align}}

We define the Clifford superalgebra $\MC_n$ by odd generators $C_1,\cdots, C_n$ with relation (\ref{clifford_rel}) and also
define the Iwahori-Hecke (super)algebra $\MHI_n$ of type A by (even) generators $T_1,\cdots,T_{n-1}$ with relations (\ref{iwahori_rel}).
By ~\cite[Theorem 2.2]{BK}, natural superalgebra 
homomorphisms 
\begin{align*}
\alpha_A:F[X_1^{\pm1},\cdots,X_n^{\pm1}]\longrightarrow \MH_n,
\quad\alpha_B:\MC_n\longrightarrow \MH_n,
\quad\alpha_C:\MHI_n\longrightarrow \MH_n
\end{align*}
are all injective and we have the following isomorphism as vector superspaces. 
\begin{align}
F[X_1^{\pm1},\cdots,X_n^{\pm1}]\otimes \MC_n\otimes \MHI_n\ISOM \MH_n,
\quad x\otimes c\otimes t\longmapsto \alpha_A(x)\alpha_B(c)\alpha_C(t).
\label{str}
\end{align}
In the sequel, we identify $f\in F[X_1^{\pm1},\cdots,X_n^{\pm1}]$ with $\alpha_A(f)\in \MH_n$ and 
omit $\alpha_A$, etc.
By (\ref{str}), we easily see that the sequence of natural superalgebra homomorphisms
\begin{align*}
\MH_0\longrightarrow \MH_1\longrightarrow \MH_2\longrightarrow\cdots
\end{align*}
are all injective and it forms a tower of superalgebras.
We also see that for each composition $\mu=(\mu_1,\cdots,\mu_\alpha)$ of $n$, 
the parabolic subsuperalgebra $\MH_\mu$ generated by 
\begin{align*}
\{X_i^{\pm1},C_i\mid 1\leq i\leq n\}\cup\bigcup_{k=1}^{\alpha-1}\{ T_j \mid \mu_1+\cdots+\mu_{k}\leq j< \mu_1+\cdots+\mu_{k+1}\}
\end{align*}
in $\MH_n$ is isomorphic to $\MH_{\mu_1}\otimes\cdots\otimes\MH_{\mu_\alpha}$ as superalgebras.

\subsection{Automorphism and antiautomorphism}

It is easily checked that there exist an automorphism $\sigma$ of $\MH_n$
and an antiautomorphism $\tau$ of $\MH_n$ defined by
\begin{align*}
\sigma &: 
T_i\longmapsto -T_{n-i}+\xi,\quad
C_j\longmapsto C_{n+1-j},\quad
X_j\longmapsto X_{n+1-j},\\
\tau &: 
T_i\longmapsto T_i+\xi C_iC_{i+1},\quad
C_j\longmapsto C_{j},\quad
X_j\longmapsto X_{j}
\end{align*}
for $1\leq i\leq n-1$ and $1\leq j\leq n$~\cite[\S2-i]{BK}.

Let $M$ be an $\MH_n$-supermodule.
The dual space $M^*$ has again an $\MH_n$-supermodule structure by
$(hf)(m)=f(\tau(h)m)$ for $f\in M^*,m\in M$ and $h\in\MH_n$.
We denote this $\MH_n$-supermodule by $M^\tau$.
We also denote by $M^\sigma$ the $\MH_n$-supermodule obtained by
twisting the action of $\MH_n$ through $\sigma$.
Then we have the following~\cite[Lemma 2.9,Theorem 2.14]{BK}.

\begin{Lem}
Let $M$ be an $\MH_m$-supermodule and let $N$ be an $\MH_n$-supermodule.
Then we have the followings. 
\begin{enumerate}
\item $(\IND^{\MH_{m+n}}_{\MH_{m,n}}M\otimes N)^{\sigma}\cong \IND^{\MH_{m+n}}_{\MH_{n,m}}N^{\sigma}\otimes M^{\sigma}$.
\item $(\IND^{\MH_{m+n}}_{\MH_{m,n}}M\otimes N)^{\tau}\cong \IND^{\MH_{m+n}}_{\MH_{n,m}}N^{\tau}\otimes M^{\tau}$.
\label{sigma_tau2}
\end{enumerate} 
Moreover, if $M$ and $N$ are both irreducible,
the same holds for $\MARU$ in place of $\otimes$.
\label{sigma_tau}
\end{Lem}

\subsection{Cartan subsuperalgebra $\MA_n$}
\label{cartan}
The subsuperalgebra 
\begin{align*}
\MA_n\DEF \langle X_i^{\pm},C_i\rangle_{1\leq i\leq n}(\subseteq \MH_n)
\end{align*}
plays a role of ``Cartan subalgebra'' in the rest of the paper.

\begin{Def}
For each integer $i\in\Z$, we define
\begin{align*}
q(i)=2\cdot\frac{q^{2i+1}+q^{-(2i+1)}}{q+q^{-1}},\quad b_{\pm}(i)=\frac{q(i)}{2}\pm\sqrt{\frac{q(i)^2}{4}-1}
\end{align*}
and choose a subset $I_q\subseteq\Z$ such that 
the map $I_q\to \{q(i)\mid i\in\mathbb{Z}\}, i\mapsto q(i)$ gives a bijection.
An $\MA_n$-supermodule $M$ is called integral if the set of eigenvalues of $X_j+X_j^{-1}$ is a subset of  
$\{q(i)\mid i\in I_q\}$ for all $1\leq j\leq n$.
Let $\mu$ be a composition of $n$.
An $\MH_\mu$-supermodule $M$ is called integral if $\RES^{\MH_{\mu}}_{\MA_n}M$ is integral. 
\end{Def}

We denote the full subcategory of $\SMOD{\MA_n}$ (resp.\ $\SMOD{\MH_\mu}$) consisting of 
integral representations by $\REP\MA_n$ (resp.\ $\REP\MH_\mu$).
We also denote by $\CH_{\mu}$ the induced $\Z$-linear homomorphism by the restriction functor $\RES^{\MH_\mu}_{\MA_n}$
\begin{align*}
\CH_{\mu}:\KKK(\REP\MH_{\mu})\longrightarrow \KKK(\REP\MA_n)
\end{align*}
between the Grothendieck groups.
We always write $\CH$ instead of $\CH_{n}$ and call $\CH M$ the formal character of $\MH_n$-supermodule $M$.

We recall a special case of covering modules~\cite[\S4-h]{BK}.

\begin{Def}
Let $m\geq 1$ and let $i\in I_q$. We define a $2m$-dimensional $\MH_1$-supermodule $L^{\pm}_m(i)$ with
an even basis $\{w_1,\cdots,w_m\}$ and an odd basis $\{w'_1,\cdots,w'_m\}$ and the following matrix representations
of actions of generators with respect to this basis.
\begin{align*}
X_1 : \begin{pmatrix}
J(b_{\pm}(i);m) & O \\
O & J(b_{\pm}(i);m)^{-1}
\end{pmatrix},\quad
C_1 : \begin{pmatrix}
O & E_m \\
E_m & O
\end{pmatrix}.
\end{align*}
Here $J(\alpha;m)\DEF(\delta_{i,j}\alpha+\delta_{i,j+1})_{1\leq i,j\leq m}$ 
stands for the Jordan matrix of size $m$.
\label{cover}
\end{Def}

We also define for $m\geq 1$ an $\MH_1$-homomorphisms $g^{\pm}_{m}:L^{\pm}_{m+1}(i)\twoheadrightarrow L^{\pm}_{m}(i)$ by
\begin{align*}
w_{k}\longmapsto 
\begin{cases}
w_{k} & \textrm{if $1\leq k\leq m$}, \\
0 & \textrm{if $k=m+1$},
\end{cases}\quad
w'_{k}\longmapsto 
\begin{cases}
w'_{k} & \textrm{if $1\leq k\leq m$}, \\
0 & \textrm{if $k=m+1$}.
\end{cases}
\end{align*}
Here $w_k$ and $w'_k$ in the left hand side are those of $L^{\pm}_{m+1}(i)$ whereas $w_k$ and $w'_k$ in the right hand side are those of $L^{\pm}_{m}(i)$.
Note that there is an odd isomorphism $g^{\circ}_{m}:L^{+}_m(i)\ISOM L^{-}_m(i)$ 
since
$J(b_{+}(i);m)$ and $J(b_{-}(i);m)^{-1}$ are similar.
For convenience, we abbreviate $L^{+}_m(i)$ (resp.\ $L^{+}_1(i)$) to $L_m(i)$ (resp.\ $L(i)$) and 
$g_m^{+}$ to $g_m$.

\begin{Def}
For $i\in I_q$ we define an $\MH_1$-supermodule $R_m(i)=\MH_1/N(i)$ where $N(i)$ is a
two-sided ideal generated by 
\begin{align*}
f(i) = 
\begin{cases}
(X_1+X_1^{-1}-q(i))^m & \textrm{if $q(i)\ne \pm 2$}, \\
(X_1-b_+(i))^m(=(X_1-b_-(i))^m) & \textrm{if $q(i)=\pm 2$}.
\end{cases}
\end{align*}
\end{Def}

As in ~\cite[\S4-h]{BK} (or by elementary linear algebra), we have the following.
\begin{Lem}
Let $i\in I_q$.
\begin{enumerate}
\item If $q(i)\ne \pm 2$, then 
there exists an even isomorphism $R_m(i)\EISO L^{+}_m(i)\oplus L^{-}_{m}(i)$ for $m\geq 1$
which commutes with the obvious surjection $R_m(i)\twoheadleftarrow R_{m+1}(i)$.
\begin{align}
\xymatrix{
R_1(i) \ar@{-}[d]^{\wr} & R_2(i) \ar@{->>}[l] \ar@{-}[d]^{\wr} & R_3(i) \ar@{->>}[l] \ar@{-}[d]^{\wr} & \cdots \ar@{->>}[l] \\
L_1(i)\oplus \Pi L_{1}(i) & L_2(i)\oplus \Pi L_{2}(i) \ar@{->>}[l]_{g_1+\Pi g_1} &  
L_3(i)\oplus \Pi L_{3}(i)\ar@{->>}[l]_{g_2+\Pi g_2} & {\cdots}. \ar@{->>}[l]
}
\label{comm_diag1}
\end{align}
\item If $q(i)=\pm 2$, then we have $R_m(i)\EISO L^+_m(i)=L^-_m(i)$ and 
there exist odd involutions $g^{\circ}_{k}$ for $k\geq 1$
make the following diagram commutes.
\begin{align}
\xymatrix{
R_1(i) \ar@{-}[d]^{\wr} & R_2(i) \ar@{->>}[l] \ar@{-}[d]^{\wr} & R_3(i) \ar@{->>}[l] \ar@{-}[d]^{\wr} & \cdots \ar@{->>}[l] \\
L_1(i) \ar@(dr,dl)^{g^{\circ}_{1}} & L_2(i) \ar@(dr,dl)^{g^{\circ}_{2}} \ar@{->>}[l]  & L_3(i) \ar@(dr,dl)^{g^{\circ}_{3}} \ar@{->>}[l]  & {\cdots}. \ar@{->>}[l] 
}
\label{comm_diag2}
\end{align}
\end{enumerate}
\end{Lem}

In virtue of $\MA_n\cong \MA_1^{\otimes n}$ and (\ref{K_isom}),
we have the following 
(see ~\cite[Lemma 4.8]{BK}).
\begin{Lem}
We have $\IRR(\REP\MA_n)=\{L(i_1)\MARU\cdots\MARU L(i_n)\mid(i_1,\cdots,i_n)\in I_q^n\}$.
Note that for $(i_1,\cdots,i_n)\in I_q^n$, $L(i_1)\MARU\cdots\MARU L(i_n)$ is of type $\TQ$
if and only if $\#\{1\leq k\leq n\mid q(i_k)=\pm 2\}$ is odd.
\end{Lem}

\subsection{Block decomposition}

The (super)center $Z(\MH_n)$ of $\MH_n$ is
naturally identified with the algebra of symmetric polynomials of $X_1+X_1^{-1},\cdots,X_n+X_n^{-1}$~\cite[Proposition 3.2(b)]{JN},~\cite[Theorem 2.3]{BK} via
\begin{align*}
F[X_1+X_1^{-1},\dots,X_n+X_n^{-1}]^{\SYM{n}}\ISOM Z(\MH_n),\quad f\longmapsto f.
\end{align*}
Thus, for any $M\in\REP\MH_n$, we have a decomposition 
$M=\bigoplus_{\gamma\in I_q^n/\SYM{n}}M[\gamma]$ with
\begin{align*}
M[\gamma]=
\{m\in M\mid \forall f\in Z(\MH_n),\exists N\in\Z_{>0},(f-\chi_\gamma(f))^Nm=0\}
\end{align*}
in $\REP\MH_n$.
Here $\chi_\gamma$ is a central character attached for $\gamma = [(\gamma_1,\cdots,\gamma_n)]$ by
\begin{align*}
\chi_\gamma:Z(\MH_n)\longrightarrow F,\quad
f(X_1+X_1^{-1},\cdots,X_n+X_n^{-1})\longmapsto f(q(\gamma_1),\cdots,q(\gamma_n)).
\end{align*}
Note that if $\gamma_1\ne\gamma_2$ in $I_q^n/\SYM{n}$, then $\chi_{\gamma_1}\ne\chi_{\gamma_2}$.

\begin{Def}
Let $M\in \IRR(\REP\MH_n)$. Then there exists a unique $\gamma\in I_q^n/\SYM{n}$ such that $M=M[\gamma]$.
In this case, we say that $M$ belongs to the block $\gamma$.
\end{Def}

We remark that this terminology coincides with the usual notion of block
since the set $\{\chi_{\gamma}\mid \gamma\in I_q^n/\SYM{n}\}$
exhausts the possible central characters arising from $\REP\MH_n$.
In fact, for any $\gamma = [(\gamma_1,\cdots,\gamma_n)]\in I_q^n/\SYM{n}$,
all the composition factors of $\IND^{\MH_n}_{\MA_n}L(\gamma_1)\MARU\cdots\MARU L(\gamma_n)$ belongs to
$\gamma$ since we have
\begin{align*}
\CH \IND^{\MH_n}_{\MA_n} L(i_1)\MARU\cdots\MARU L(i_n)
=\sum_{w\in \SYM{n}}[L(i_{w(1)})\MARU\cdots\MARU L(i_{w(n)})].
\end{align*}
This identity~\cite[Lemma 4.10]{BK} follows from the Mackey theorem\cite[Theorem 2.8]{BK}.

\subsection{Kashiwara operators}
Recall the Kato supermodule $L(i^n)\DEF\IND^{\MH_n}_{\MA_n}L(i)^{\MARU n}$~\cite[\S4-g]{BK}.
Using them, we can introduce
Kashiwara operators $\Te_i$ and $\Tf_i$ that send an irreducible supermodule to another one (if defined).
We first recall a fundamental property of Kato's modules~\cite[Theorem 4.16.(i)]{BK}.
\begin{Thm}
For $i\in I_q$ and $n\geq 1$, $L(i^n)$ is irreducible of the same type as $L(i)^{\MARU n}$ and it is the only irreducible
supermodule in its block of $\REP \MH_n$.
\label{kato}
\end{Thm}

\begin{Def}
For $i\in I_q, 0\leq m\leq n$ and $M\in\REP\MH_n$, we denote by $\Delta_{i^m}M$ the
simultaneous generalized $q(i)$-eigenspace of the commuting operators $X_k+X_k^{-1}$ for all $n-m<k\leq n$.
Note that $\Delta_{i^m}M$ is an $\MH_{n-m,m}$-supermodule.
We also define $\varepsilon_i(M)=\max\{m\geq 0\mid \Delta_{i^m}M\ne 0\}$.
\end{Def}

By ~\cite[\S5-a]{BK}, we have the followings~\cite[Lemma 5.5, Theorem 5.6, Corollary 5.8]{BK}.
\begin{Thm}
Let $i\in I_q, m\geq 0$ and $M\in\IRR(\REP\MH_n)$.
\begin{enumerate}
\item $N\DEF\HEAD\IND_{\MH_{n,m}}^{\MH_{n+m}}M\MARU L(i^m)$ is irreducible with $\varepsilon_i(N)=\varepsilon_i(M)+m$ and
any other irreducible composition factor $L$ of $\IND_{\MH_{n,m}}^{\MH_{n+m}}M\MARU L(i^m)$ satisfies
$\varepsilon_i(L)<\varepsilon_i(M)+n$.
\label{Kato_irr_thm_1st}
\item Assume that $0\leq m\leq \varepsilon_i(M)$. 
There exists (up to isomorphism) an irreducible $\MH_{n-m}$-supermodule $L$ such that
$\TYPE L=\TYPE M, \varepsilon_i(L)=\varepsilon_i(M)-m$ and $\SOC\Delta_{i^m}M\cong L\MARU L(i^m)$.
\label{type_thm}
\item Assume that $\varepsilon_i(M)>0$. Then we have
\begin{align*}
\SOC\RES^{\MH_{n-1,1}}_{\MH_{n-1}}\Delta_i(M)\EISO
\begin{cases}
L\oplus\Pi L & \textrm{if $\TYPE M=\TQ$ or $q(i)\ne \pm 2$}, \\
L & \textrm{if $\TYPE M=\TM$ and $q(i)=\pm 2$},
\end{cases}
\end{align*}
for some irreducible $\MH_{n-1}$-module $L$ of the same type as $M$ if $q(i)\ne \pm 2$ and
of the opposite type to $M$ if $q(i)=\pm 2$.
\label{res_socle}
\end{enumerate} 
\label{Kato_irr_thm}
\end{Thm}

\begin{Def}
Let us write $B(\infty)\DEF\bigsqcup_{n\geq 0}\IRR(\REP\MH_n)$.
For $i\in I_q$, we define maps $\Te_i,\Tf_i:B(\infty)\sqcup\{0\}\to B(\infty)\sqcup\{0\}$ as follows.
\begin{itemize}
\item $\Te_i0=\Tf_i0=0$.
\item For $M\in\IRR(\REP\MH_n)$, we set $\Tf_i M=\HEAD\IND_{\MH_{n,1}}^{\MH_{n+1}}M\MARU L(i)$.
\item For $M\in\IRR(\REP\MH_n)$, we set $\Te_i M=0$ if $\varepsilon_i(M)=0$ otherwise
$\Te_iM=L$ for a unique $L\in\IRR(\REP\MH_{n-1})$ with $\SOC\Delta_{i}M\cong L\MARU L(i)$.
\end{itemize}
\label{kashiwara_def}
\end{Def}

Note that we have $\varepsilon_i(M)=\max\{m\geq 0\mid (\Te_i)^mM\ne 0\}$ by Theorem \ref{Kato_irr_thm} (\ref{type_thm}).
By ~\cite[Lemma 5.10]{BK}, $\Te_i$ and $\Tf_i$ satisfy one of the axioms 
of Kashiwara's crystal (see Definition \ref{def_of_crystal} (\ref{def_of_crystal5})), i.e.,
\begin{Lem}
For $M,N\in B(\infty)$ and $i\in I_q$,
$\Tf_iM=N$ is equivalent to $\Te_iN=M$.
\label{crystal1}
\end{Lem}

\begin{Def}
For $\BII=(i_1,\cdots,i_n)\in I_q^n$, we define
$
L(\BII)=\Tf_{i_n}\Tf_{i_{n-1}}\cdots\Tf_{i_2}\Tf_{i_1}\TRIVREP$.
Here $\TRIVREP$ is the trivial representation of $\MH_0=F$.
\label{def_of_L}
\end{Def}

Note that $L(\BII)$ applied for $\BII=(i,\cdots,i)$ 
coincides with the Kato supermodule $L(i^n)$ by Theorem \ref{kato}.
By an inductive use of Lemma \ref{crystal1}, 
we have the following~\cite[\S5-d, Lemma 5.15]{BK}.
\begin{Cor}
For any $L\in\IRR(\REP\MH_n)$ there exists $\BII\in I_q^n$ such that
$L\cong L(\BII)$. $\RES^{\MH_n}_{\MA_n}L(\BII)$ has a submodule isomorphic 
to $L(i_1)\MARU\cdots\MARU L(i_n)$.
\label{label_of_irr}
\end{Cor}

Also a repeated use of 
Theorem \ref{Kato_irr_thm} (\ref{type_thm}) implies the following~\cite[Lemma 5.14]{BK}. 
\begin{Cor}
Let $M\in\IRR(\REP\MH_n)$ and let $\mu$ be a composition of $n$.
For any irreducible composition factor $N$ of $\RES^{\MH_n}_{\MH_{\mu}}M$, we have $\TYPE M = \TYPE N$.
\label{type_of_irr}
\end{Cor}

\subsection{Root operators}
\label{root_aff}
We shall define root operators $e_i$ as a direct summand of $\RES^{\MH_{n-1,1}}_{\MH_{n-1}}\Delta_i$.
Note that for any $M\in\REP\MH_n$ and $i\in I_q$, we have a natural identification
\begin{align}
\RES^{\MH_{n-1,1}}_{\MH_{n-1}}\Delta_i M
\EISO
\varinjlim_{m}\HOM_{\MH'_1}(R_m(i),M).
\label{natural_id}
\end{align}
Here $\MH'_1$ stands for a subsuperalgebra in $\MH_n$ generated by $\{X^{\pm 1}_n,C_n\}$
isomorphic to $\MH_1$. 
Considering (\ref{comm_diag1}) or (\ref{comm_diag2}), we can chose a
summand of $\RES^{\MH_{n-1,1}}_{\MH_{n-1}}\Delta_iM$ appropriately as follows.

\begin{Def}
For $M\in\IRR(\REP\MH_n)$ and $i\in I_q$, we define 
\begin{align*}
e_i M &= \varinjlim_{m}\OHOM_{\MH'_1}((L_m(i),\theta_m^\circ),(M,\theta_M))(\in \REP\MH_{n-1}).
\end{align*}
Here the $\theta$'s are defined as follows.
\begin{itemize}
\item
$\theta^\circ_m=\ID_{L_m(i)}$ if $q(i)\ne\pm 2$, and $\theta_m^\circ=g_m^\circ$ otherwise.
\item
$\theta_M=\ID_M$ if $\TYPE M=\TM$, and $\theta_M$ is an odd involution of $M$ otherwise.
\end{itemize}
\label{def_of_e}
\end{Def}

Thus, by Theorem \ref{Kato_irr_thm} (\ref{res_socle}), we have 
\begin{align*}
\RES^{\MH_{n-1,1}}_{\MH_{n-1}}\Delta_i(M)\EISO
\begin{cases}
e_iM & \textrm{if $\TYPE M=\TM$ and $q(i)\pm 2$}, \\
e_iM\oplus \Pi e_iM & \textrm{if $\TYPE M=\TQ$ or $q(i)\ne \pm 2$}.
\end{cases}
\end{align*}

By the commutativity of $\RES^{\MH_n}_{\MH_{n-1}}$ and $\tau$-duality, we see the following~\cite[Lemma 6.6.(i)]{BK}.
\begin{Cor}
Let $M\in\IRR(\REP\MH_n)$ and $i\in I_q$. Then $e_iM$ is non-zero if and only if $\Te_iM$ is non-zero,
in which case $e_iM$ is a self-dual indecomposable module with irreducible socle and cosocle
isomorphic to $\Te_iM$.
\label{self_dual}
\end{Cor}

Also, as seen in ~\cite[\S6-d]{BK}, we have the followings~\cite[Theorem 6.11]{BK}.
\begin{Thm}
Let $M\in\IRR(\REP\MH_n)$ and $i\in I_q$.
\begin{enumerate}
\item In $\KKK(\REP\MH_n)$, we have $[e_iM]=\varepsilon_i(M)[\Te_iM]+\sum c_a[N_a]$ where
$N_a$ are irreducibles with $\varepsilon_i(N_a)<\varepsilon_i(M)-1$.
\label{root_op1}
\item If $q(i)\ne\pm2$, then $\varepsilon_i(M)$ is the maximal size of a Jordan block of $X_n+X^{-1}_n$
on $M$ with eigenvalue $q(i)$.
\label{root_op2}
\item If $q(i)=\pm2$, then $\varepsilon_i(M)$ is the maximal size of a Jordan block of \ $X_n$ 
on $M$ with eigenvalue $b_+(i)=b_-(i)$.
\label{root_op22}
\item $\END_{\MH_{n-1}}(e_iM)\EISO \END_{\MH_{n-1}}(\Te_iM)^{\oplus\varepsilon_i(M)}$ as vector superspaces.
\label{root_op3}
\end{enumerate}
\label{root_op}
\end{Thm}

\subsection{Kashiwara's crystal structure}
\label{general_setting}
In this subsection,
let $A=(a_{ij})_{i,j\in I_q}$ be an arbitrary symmetrizable generalized Cartan matrix indexed by $I_q$.
We identify $I_q^n/\SYM{n}$ and 
$\Gamma_n\DEF \{\sum_{i\in I_q}k_i\alpha_i\in\sum_{i\in I_q}\Z_{\geq 0}\alpha_i\mid \sum_{i\in I_q} k_i=n\}$ by
\begin{align*}
b_A:I_q^n/\SYM{n} &\ISOM \Gamma_n,\quad
[(\gamma_1,\cdots,\gamma_n)] \longmapsto \sum_{k=1}^{n}\alpha_{\gamma_{k}}.
\end{align*}
For $M\in\IRR(\REP\MH_n)$ belonging to a block $\gamma\in I_q^n/\SYM{n}$ and $i\in I_q$, 
we define
\begin{align*}
\WT(M)=-b_A(\gamma),\quad
\varphi_i(M)=\varepsilon_i(M)+\langle h_i, \WT(M)\rangle.
\end{align*}
By Theorem \ref{Kato_irr_thm} and Lemma \ref{crystal1}, we can check the following~\cite[Lemma 8.5]{BK}.
\begin{Lem}
The 6-tuple
$(B(\infty),\WT,\{\varepsilon_i\}_{i\in I_q}, \{\varphi_i\}_{i\in I_q},\{\Te_i\}_{i\in I_q},\{\Tf_i\}_{i\in I_q})$
is a $\GEE(A)$-crystal.
\label{crystal_inf}
\end{Lem}

Finally, we introduce $\sigma$-version of the above operations for $M\in B(\infty)$ and $i\in I_q$.
\begin{align*}
\Te^*_iM=(\Te_i(M^\sigma))^\sigma,\quad
\Tf^*_iM=(\Tf_i(M^\sigma))^\sigma,\quad
\varepsilon^*_i(M)=\varepsilon_i(M^\sigma).
\end{align*}

Of course, we have $\varepsilon^*_i(M)=\max\{k\geq 0\mid (\Te^*_i)^kM\ne 0\}$. However
$\varepsilon^*_i(M)$ has another description as follows by Theorem \ref{root_op} (\ref{root_op2}) and Theorem \ref{root_op} (\ref{root_op22}).
\begin{Lem}
Let $i\in I_q$ and $M\in\IRR(\REP \MH_n)$. 
\begin{itemize}
\item If $q(i)\ne \pm 2$, then $\varepsilon^*_i(M)$ is the maximal size of a Jordan
block of $X_1+X^{-1}_1$ on $M$ with eigenvalue $q(i)$.
\item If $q(i)= \pm 2$, then $\varepsilon^*_i(M)$ is the maximal size of a Jordan
block of $X_1$ on $M$ with eigenvalue $b_+(i)=b_-(i)$.
\end{itemize}
\label{kill_pr}
\end{Lem}

We also quote results concerning the commutativity of $\Te_i$ and $\Tf_j^*$\cite[Lemma 8.1, Lemma 8.2, Lemma 8.4]{BK}.
\begin{Lem}
Let $M\in\IRR(\REP \MH_n)$ and $i, j\in I_q$.
\begin{enumerate}
\item $\varepsilon_i(\Tf^*_iM)=\varepsilon_i(M)$ or $\varepsilon_i(\Tf^*_iM)=\varepsilon_i(M)+1$.
\label{comm_cry1}
\item If $i\ne j$, then $\varepsilon_i(\Tf^*_jM)=\varepsilon_i(M)$.
\label{comm_cry2}
\item If $\varepsilon_i(\Tf^*_jM)=\varepsilon_i(M)$ (denoted by $\varepsilon$), 
then $\Te_i^{\varepsilon}\Tf^*_jM\cong \Tf^*_j\Te_i^{\varepsilon}M$.
\label{comm_cry3}
\item If $\varepsilon_i(\Tf^*_iM)=\varepsilon_i(M)+1$, then $\Te_i\Tf^*_iM\cong M$.
\label{comm_cry4}
\end{enumerate}
\label{comm_cry}
\end{Lem}

\subsection{Hopf algebra structure}

Consider the graded $\Z$-free module 
\begin{align*}
K(\infty)=\bigoplus_{n\geq 0}\KKK(\REP\MH_n)
\end{align*}
with natural basis $B(\infty)$ and define $\Z$-linear maps
\begin{align*}
\diamond_{m,n} &: \KKK(\REP\MH_m)\otimes\KKK(\REP\MH_n)\ISOM \KKK(\REP\MH_{m,n}) 
\xrightarrow[]{\IND^{\MH_{m+n}}_{\MH_{m,n}}} \KKK(\REP\MH_{m+n}), \\
\Delta_{m,n} &: \KKK(\REP\MH_{m+n}) 
\xrightarrow[]{\RES^{\MH_{m+n}}_{\MH_{m,n}}} \KKK(\REP\MH_{m,n})\ISOM \KKK(\REP\MH_m)\otimes\KKK(\REP\MH_n), \\
\diamond &= \sum_{m,n\geq 0}\diamond_{m,n}:K(\infty)\otimes K(\infty)\longrightarrow K(\infty),
\quad \iota : \Z\ISOM \KKK(\REP\MH_{0})\stackrel{\textrm{inj}}{\longhookrightarrow}\KKK(\infty) \\
\Delta &= \sum_{m,n\geq 0}\Delta_{m,n}:K(\infty)\longrightarrow K(\infty)\otimes K(\infty),\quad
\varepsilon:\KKK(\infty)\stackrel{\textrm{proj}}{\longtwoheadrightarrow} \KKK(\REP\MH_{0})\ISOM\Z.
\end{align*}
Note that $\diamond_{m,n}$ is well-defined since for any $M\in \REP\MH_{m,n}$ 
we have $\IND^{\MH_{m+n}}_{\MH_{m,n}}M\in \REP\MH_{m+n}$ by ~\cite[Lemma 4.6]{BK}.

Transitivity of induction and restriction makes $(\KKK(\infty),\diamond,\iota)$  a
graded $\Z$-algebra and $(\KKK(\infty),\Delta,\varepsilon)$  a graded $\Z$-coalgebra.
Injectivity of the formal character map $\CH:\KKK(\REP\MH_n)\hookrightarrow\KKK(\REP\MA_n)$~\cite[Theorem 5.12]{BK} implies
$L\cong L^\tau$ for all $L\in B(\infty)$~\cite[Corollary 5.13]{BK}.
Combine it with Lemma \ref{sigma_tau} (\ref{sigma_tau2}),
we see that the multiplication of $(\KKK(\infty),\diamond,\iota)$ is commutative. 
By Mackey theorem~\cite[Theorem 2.8]{BK}, we see that $(K(\infty),\diamond,\Delta,\iota,\varepsilon)$ is a 
graded $\Z$-bialgebra\footnote{In checking the details, we need the commutativity of the following 
diagrams for $m\geq k$ and $n\geq l$ and it follows from Corollary \ref{type_of_irr}.
\begin{align*}
\xymatrix{
\KKK(\REP\MH_{m,n}) \ar@{-}[r]^{\!\!\!\!\!\!\!\!\!\!\!\!\!\!\!\!\sim} \ar[d]^{\RES^{\MH_{m,n}}_{\MH_{k,l}}} 
\ar@{}[dr]
& \KKK(\REP\MH_m)\otimes\KKK(\REP\MH_n) \ar[d]^{{\RES^{\MH_m}_{\MH_k}}\otimes{\RES^{\MH_n}_{\MH_l}}} \\
\KKK(\REP\MH_{k,l})\ar@{-}[r]^{\!\!\!\!\!\!\!\!\!\!\!\!\!\!\!\!\sim} & \KKK(\REP\MH_k)\otimes\KKK(\REP\MH_l),
}
\xymatrix{
\KKK(\REP\MH_{m,n}) \ar@{-}[r]^{\!\!\!\!\!\!\!\!\!\!\!\!\!\!\!\!\sim} \ar@{}[dr]
& \KKK(\REP\MH_m)\otimes\KKK(\REP\MH_n) \\
\KKK(\REP\MH_{k,l})\ar@{-}[r]^{\!\!\!\!\!\!\!\!\!\!\!\!\!\!\!\!\sim} \ar[u]_{\IND^{\MH_{m,n}}_{\MH_{k,l}}}  
& \KKK(\REP\MH_k)\otimes\KKK(\REP\MH_l) \ar[u]_{{\IND^{\MH_m}_{\MH_k}}\otimes{\IND^{\MH_n}_{\MH_l}}}.
}
\end{align*}
}.
Since a connected (non-negatively) graded bialgebra is a Hopf algebra~\cite[pp.238]{Swe},
we get the following~\cite[Theorem 7.1]{BK}.

\begin{Thm}
$(K(\infty),\diamond,\Delta,\iota,\varepsilon)$ is a commutative graded Hopf algebra over $\Z$.
Thus, $K(\infty)^*$ is a cocommutative graded Hopf algebra over $\Z$.
\end{Thm}

Here $K(\infty)^*$ is a graded dual of $K(\infty)$, i.e., $K(\infty)^*=\bigoplus_{n\geq 0}\HOM_\Z(\KKK(\REP\MH_n),\Z)$.
$K(\infty)^*$ has a natural $\Z$-free basis $\{\delta_M\mid M\in B(\infty)\}$ defined by 
$\delta_M([M])=1$ and $\delta_M([N])=0$ for all $[N]\in B(\infty)$ with $N\not\cong M$.

\subsection{Left $K(\infty)^*$-module structure on $K(\infty)$}
\label{module_str_inf}
By ~\cite[Proposition 2.1.1]{Swe}, for a coalgebra $C$ and a right $C$-comodule $\omega:M\to M\otimes C$, 
$M$ is turned into a left $C^*$-module by
\begin{align*}
C^*\otimes M
\xrightarrow[]{\ID_{C^*}\otimes \omega} C^*\otimes M\otimes C
\xrightarrow[]{\SWAP\otimes \ID_C} M\otimes C^*\otimes C
\xrightarrow[]{\ID_M \otimes\langle,\rangle} M\otimes \Z\ISOM M.
\end{align*}

It implies that each coalgebra $C$ is naturally regarded as a left $C^*$-module.
It is easily seen that if $C$ is connected (non-negatively) graded coalgebra then
the left action of $C^*$ is faithful.
Thus, $K(\infty)$ has a natural faithful left $K(\infty)^*$-module structure 
and it coincides with the root operators $e_i$ in the following sense~\cite[Lemma 7.2, Lemma 7.4]{BK}.

\begin{Lem}
For $i\in I_q, r\geq 1$ and $M\in K(\infty)$, we have $\delta_{L(i^r)}\cdot M=e^{(r)}_iM$.
\label{divided_root_op}
\end{Lem}

%

Note that $e_i^{(r)}$ is a priori an operator on $K(\infty)_\Q\DEF\Q\otimes K(\infty)$, however
as seen in Lemma \ref{divided_root_op} it is a well-defined operator on $K(\infty)$.
We can prove it directly by defining a divided power root operators $e^{(r)}_i$ in a 
module-theoretic way~\cite[\S6-c]{BK}.

\section{Cyclotomic Hecke-Clifford superalgebra}
\label{cyc_quot_section}
\subsection{Definition and vector superspace structure}
\begin{Def}
Let $n\geq 1$ and assume that $\EF=a_dX_1^d+\cdots+a_0\in F[X_1](\subseteq\MH_n)$ satisfies 
$C_1\EF=a_0X_1^{-d}\EF C_1$ (equivalently
saying, the coefficients $\{a_i\}_{i=0}^d$ of $\EF$ satisfies $a_d=1$ and $a_i=a_0a_{d-i}$ for all $0\leq i\leq d$).
We define the cyclotomic Hecke-Clifford superalgebra $\MH^{\EF}_n=\MH_n/\langle \EF \rangle$ for $n\geq 1$
and define $\MH^\EF_0=F$.
\label{cyc_quo}
\end{Def}

Note that the antiautomorphism $\tau$ of $\MH_n$ induces an
anti-automorphism of $\MH^\EF_n$ also written by $\tau$.
As in the affine case, for an $\MH^\EF_n$-supermodule $M$
we write $M^\tau$ the dual space $M^*$ with $\MH^\EF_n$-supermodule structure obtained by $\tau$.

By ~\cite[Theorem 3.6]{BK},
$\MH^{\EF}_n$ is a finite-dimensional superalgebra whose basis is a canonical images of the elements
\begin{align*}
\{X_1^{\alpha_1}\cdots X_n^{\alpha_n}C_1^{\beta_1}\cdots C_n^{\beta_n}T_w\mid 
0\leq\alpha_k<d,\beta_k\in\FF,w\in\SYM{n}\}.
\end{align*}
Thus, we have
the following commutativity between towers of superalgebras.
\begin{align*}
\xymatrix{
\MH_0 \ar@{^{(}->}[r] \ar@{->>}[d] & \MH_1 \ar@{^{(}->}[r] \ar@{->>}[d] & \MH_2 \ar@{^{(}->}[r] \ar@{->>}[d] & {\cdots} \\
\MH^\EF_0 \ar@{^{(}->}[r] & \MH^\EF_1 \ar@{^{(}->}[r] & \MH^\EF_2 \ar@{^{(}->}[r] & {\cdots}.
}
\end{align*}

It makes us possible to define inductions and restrictions for $\{\MH^\EF_n\}_{n\geq 0}$ as well as
$M^\tau$ and we have the following~\cite[Theorem 3.9, Corollary 3.15]{BK}.
\begin{Thm}
Let $M$ be an $\MH^\EF_n$-supermodule. 
\begin{enumerate}
\item There is a natural isomorphism of $\MH^\EF_n$-modules.
\begin{align*}
\RES^{\MH^\EF_{n+1}}_{\MH^{\EF}_n}\IND^{\MH^\EF_{n+1}}_{\MH^\EF_n}M
\EISO
(M\oplus \Pi M)^{d}\oplus \IND^{\MH^\EF_{n}}_{\MH^{\EF}_{n-1}}\RES^{\MH^\EF_{n}}_{\MH^{\EF}_{n-1}}M.
\end{align*}
\label{ind_res_res_ind}
\item The functors $\RES^{\MH^\EF_{n+1}}_{\MH^\EF_n}$ and $\IND^{\MH^\EF_{n+1}}_{\MH^\EF_n}$ are 
left and right adjoint to each other.
\item There is a natural isomorphism as $\MH^\EF_{n+1}$-modules 
$\IND^{\MH^\EF_{n+1}}_{\MH^{\EF}_n}(M^\tau)\EISO(\IND^{\MH^\EF_{n+1}}_{\MH^{\EF}_n}M)^\tau$.
\end{enumerate}
\label{ind_res}
\end{Thm}

We also define two natural functors.
Note that $\PR^\EF$ is a left adjoint to $\INFL^\EF$.
\begin{align*}
\PR^\EF:\SMOD{\MH_n}\longrightarrow\SMOD{\MH_n^\EF},\quad M\longmapsto M/\langle \EF \rangle M,\\
\INFL^\EF:\SMOD{\MH_n^\EF}\longrightarrow \SMOD{\MH_n},\quad M\longmapsto \RES^{\MH_n^\EF}_{\MH_n}M.
\end{align*}

\subsection{Kashiwara operators}
\label{kashiwara_cyc}
Kashiwara operators for cyclotomic superalgebras are 
defined using those defined for affine superalgebras as follows.
By Lemma \ref{crystal1}, 
$\Te^\EF_i$ and $\Tf^\EF_i$ clearly satisfy Definition \ref{def_of_crystal} (\ref{def_of_crystal5}).

\begin{Def}
Let us write $B(\EF)\DEF\bigsqcup_{n\geq 0}\IRR(\SMOD{\MH^\EF_n})$.
For $i\in I_q$, we define maps $\Te^\EF_i,\Tf^\EF_i:B(\EF)\sqcup\{0\}\to B(\EF)\sqcup\{0\}$ as follows.
\begin{itemize}
\item $\Te^\EF_i0=\Tf^\EF_i0=0$.
\item For $M\in\IRR(\SMOD{\MH^\EF_n})$, we set $\Te^\EF_iM=(\PR^\EF\circ \Te_i\circ \INFL^\EF)M$ and
$\Tf^\EF_iM=(\PR^\EF\circ \Tf_i\circ \INFL^\EF)M$.
\end{itemize}
\label{def_of_kashiwara_cyc}
\end{Def}

We also define for $M\in B(\EF)$ and $i\in I_q$,
\begin{align*}
\varepsilon^\EF_i(M) &= \max\{k\geq 0\mid (\Te^\EF_i)^k(M)\ne 0\}(=\varepsilon_i(\INFL^\EF M)),\\
\varphi^\EF_i(M)     &= \max(\{k\geq 0\mid (\Tf^\EF_i)^k(M)\ne 0\}\sqcup\{+\infty\}).
\end{align*}
Note that although $\varphi^\EF_i(M)$ may take the value $+\infty$, 
it always takes a finite value as seen in Lemma \ref{root_op2f} (\ref{limit_pt}) below.

\subsection{Root operators}
\begin{Def}
For $M\in\SMOD{\MH^\EF_n}$ such that $\INFL^\EF M$ belongs 
to a block $\gamma\in I_q^n/\SYM{n}$ with $-b_A(\gamma)=\sum_{i\in I_q}k_i\alpha_i$,
we define
\begin{align*}
\RES^\EF_i M &=
\begin{cases}
\PR^\EF((\INFL^\EF\RES^{\MH^\EF_n}_{\MH^\EF_{n-1}}M)[b^{-1}_A(-(\gamma-\alpha_i))]) & \textrm{if $k_i>0$}, \\
0 & \textrm{if $k_i=0$},
\end{cases} \\
\IND^\EF_i M &=
\PR^\EF((\INFL^\EF\IND^{\MH^\EF_{n+1}}_{\MH^\EF_{n}}M)[b^{-1}_A(-(\gamma+\alpha_i))]).
\end{align*}
In general, for $M\in\SMOD{\MH^\EF_n}$ we define $\RES^\EF_i M$ (resp.\ $\IND^\EF_i M$)
by applying $\RES^\EF_i $ (resp.\ $\IND^\EF_i $)
for each summand of $M=\bigoplus_{\gamma\in I_q^n/\SYM{n}}\PR^\EF((\INFL^\EF M)[\gamma])$.
\end{Def}

By Theorem \ref{ind_res} and central character consideration, we get the following~\cite[Lemma 6.1]{BK}.
\begin{Cor}
Let $i\in I_q$.
\begin{enumerate}
\item $\RES^\EF_i$ and $\IND^\EF_i$ are left and right adjoint to each other.
\label{proj_proj}
\item For each $M\in\SMOD{\MH^\EF_n}$ there are natural isomorphisms
\begin{align*}
\IND^\EF_i(M^\tau)\EISO(\IND^\EF_iM)^\tau,\quad
\RES^\EF_i(M^\tau)\EISO(\RES^\EF_iM)^\tau.
\end{align*}
\end{enumerate}
\label{ind_res2}
\end{Cor}

Note that $\RES^\EF_i$ is nothing but $\PR^\EF\circ\RES^{\MH_{n-1,1}}_{\MH_{n-1}}\circ\Delta_i\circ\INFL^\EF$
and it can be described as follows (see also (\ref{natural_id})). Replacing each operator with its left adjoint and
checking the well-definedness, we have the following~\cite[Lemma 6.2]{BK}.
\begin{Lem}
Let $M\in\SMOD{\MH^\EF_n}$ and $i\in I_q$.
There are natural isomorphisms 
\begin{align*}
\RES^\EF_iM &\EISO \varinjlim_{m}\PR^\EF\HOM_{\MH'_1}(R_m(i),\INFL^\EF M), \\
\IND^\EF_iM &\EISO \varprojlim_{m}\PR^\EF\IND^{\MH_{n+1}}_{\MH_{n}\otimes\MH_1}((\INFL^\EF M)\otimes R_m(i)).
\end{align*}
Here both limits are stabilized after finitely many terms.
\label{func_isom}
\end{Lem}

As in the affine case, 
we can choose a suitable
summand of $\RES^\EF_iM$ and $\IND^\EF_iM$ using (\ref{comm_diag1}) or (\ref{comm_diag2}).

\begin{Def}
Let $M\in\IRR(\SMOD{\MH^\EF_n})$. 
We define
\begin{align*}
e^\EF_i X &= \varinjlim_{m}\PR^\EF\OHOM_{\MH'_1}((L_m(i),\theta_m^\circ),(\INFL^\EF X,\INFL^\EF\theta_X)), \\
f^\EF_i X &= \varprojlim_{m}\PR^\EF\IND^{\MH_{n+1}}_{\MH_{n}\otimes\MH_1}(\INFL^\EF X,\INFL^\EF\theta_X)\MARU (L_m(i),\theta_m^\circ)
\end{align*}
for each $X=M$ or $X=P\DEF P_M$ and $i\in I_q$.
Here $\theta$'s are defined as follows.
\begin{itemize}
\item
$\theta^\circ_m=\ID_{L_m(i)}$ if $q(i)\ne\pm 2$, and $\theta_m^\circ=g_m^\circ$ otherwise.
\item
$\theta_M=\ID_M$ if $\TYPE M=\TM$, and $\theta_M$ is an odd involution of $M$ otherwise.
\item $\theta_P=\ID_{P}$ if $\TYPE M=\TM$, and $\theta_P$ is an odd involution of $P$ whose
existence is guaranteed by ~\cite[Lemma 12.2.16]{Kl2}\footnote{
In ~\cite[\S6-c]{BK}, they claim that
for $\TYPE M=\TQ$ a lift $\theta_P$ which is also an odd involution of the odd invoution $\theta_M$ is unique.
However, it is not true in general. 
Note that any odd involution of $P$ works in the rest of this paper
since our aim is to halve $\RES^\EF_i P$ or $\IND^\EF_i P$
in the same way as $\RES^\EF_i M$ or $\IND^\EF_i M$ to obtain Lemma \ref{commutativity_proj}.} otherwise.
\end{itemize}
\label{root_op_cyc}
\end{Def}

Note that for a principal indecomposable $P$ and $i\in I_q$,
$e^\EF_iP$ and $f^\EF_iP$ are again projectives since they are summands of $\RES^\EF_i$
and $\IND^\EF_i$ respectively (see also Corollary \ref{ind_res2}). 
Thus, we define operators $e^\EF_i$ and $f^\EF_i$ on $K(\EF)\DEF\bigoplus_{n\geq 0}\KKK(\SMOD{\MH^\EF_n})$ and
$K(\EF)^*\cong \bigoplus_{n\geq 0}\KKK(\PROJ\MH^\EF_n)$. 


\begin{Lem}
For any principal indecomposable $\MH^\EF_n$-supermodule $P$ and $i\in I_q$, 
we have in $\KKK(\SMOD{\MH^\EF_{n-1}})$ and $\KKK(\SMOD{\MH^\EF_{n+1}})$ respectively
\begin{align*}
e^\EF_i(\omega_{\MH^\EF_n}[P])=\omega_{\MH^\EF_{n-1}}([e^\EF_iP]),\quad
f^\EF_i(\omega_{\MH^\EF_n}[P])=\omega_{\MH^\EF_{n+1}}([f^\EF_iP]).
\end{align*}
\label{commutativity_proj}
\end{Lem}

\begin{proof}
Let $A$ and $B$ be superalgebras and consider an (even) exact functor $X:\SMOD{A}\to\SMOD{B}$ 
which sends every projective to a projective.
Then for any principal indecomposable projective $A$-supermodule $P$, we easily see
$X(\omega_{A}[P])=\omega_{B}([XP])$ in $\KKK(\SMOD{B})$. By Corollary \ref{ind_res2} (\ref{proj_proj}), it implies that 
\begin{align*}
\RES^\EF_i(\omega_{\MH^\EF_n}[P])=\omega_{\MH^\EF_{n-1}}([\RES^\EF_iP]),\quad
\IND^\EF_i(\omega_{\MH^\EF_n}[P])=\omega_{\MH^\EF_{n+1}}([\IND^\EF_iP]).
\end{align*}
We shall only show $e^\EF_i(\omega_{\MH^\EF}[P])=\omega_{\MH^\EF_{n-1}}([e^\EF P])$ in $\KKK(\SMOD{\MH^\EF_{n-1}})$ because
the other is similar.
By (\ref{comm_diag1}), (\ref{comm_diag2}), Lemma \ref{func_isom} and Definition \ref{root_op_cyc}, we have
\begin{align*}
[e^\EF_iP] = \begin{cases}
[\RES^\EF_iP] & \textrm{if $q(i)=\pm 2$ and $\TYPE\HEAD P=\TM$},\\
\frac{1}{2}[\RES^\EF_iP] & \textrm{if otherwise}
\end{cases}
\end{align*}
in $\KKK(\PROJ \MH^\EF_{n-1})$. Similarly, for  $M\in\IRR(\SMOD{\MH^\EF_{n-1}})$ we have 
\begin{align*}
[e^\EF_iM] = \begin{cases}
[\RES^\EF_iM] & \textrm{if $q(i)=\pm 2$ and $\TYPE M=\TM$},\\
\frac{1}{2}[\RES^\EF_iM] & \textrm{if otherwise}
\end{cases}
\end{align*}
in $\KKK(\SMOD{\MH^\EF_{n-1}})$. Thus, it is enough to show that 
for each irreducible factor $N$ of $P$ we have $\TYPE N=\TYPE\HEAD P$. 
Take a unique $\gamma\in I_q^n/\SYM{n}$ such that $P=P[\gamma]$. 
It is clear that $N$ also belongs to the block $\gamma$.
By Corollary \ref{type_of_irr}, 
$\TYPE N$ is determined by
its central character. 
\end{proof}

Since $e^\EF_i = \PR^\EF\circ e_i\circ \INFL^\EF$ and $\Te^\EF_i = \PR^\EF\circ \Te_i\circ \INFL^\EF$,
Corollary \ref{self_dual} and Theorem \ref{root_op} hold for $M\in\REP\MH^\EF_n$ and $i\in I_q$ by replacing $e_i$, $\Te_i$ and $\varepsilon_i$
appearing there with $e^\EF_i$, $\Te^\EF_i$ and $\varepsilon^\EF_i$ respectively.
We quote the corresponding properties of $f^\EF_i$, $\Tf^\EF_i$ and $\varphi_i^\EF$~\cite[Theorem 6.6.(ii), Lemma 6.18, Corollary 6.24]{BK}.
\begin{Lem}
Let $M\in\IRR(\SMOD{\MH^\EF_n})$ and $i\in I_q$. 
\begin{enumerate}
\item
$f^\EF_iM$ is non-zero if and only if $\Tf^\EF_iM$ is non-zero,
in which case it is a self-dual indecomposable module with irreducible socle and cosocle
isomorphic to $\Tf_iM$.
\item $\varphi^\EF_i(M)$ is the smallest $m\geq 1$ (thus, takes a finite value by Lemma \ref{func_isom}) such that 
$f^\EF_i M = \PR^\EF\IND^{\MH_{n+1}}_{\MH_{n}\otimes\MH_1}(\INFL^\EF M,\INFL^\EF\theta_M)\MARU (L_m(i),\theta_m^\circ)$ if
$f^\EF_i M\ne 0$. If $f^\EF_i M=0$ then $\varphi^\EF_i(M)=0$.
\label{limit_pt}
\item In $\KKK(\REP\MH_n)$, we have $[f^{\EF}_i M]=\varphi^\EF_i(M)[\Tf_iM]+\sum c_a[N_a]$ where
$N_a$ are irreducibles with $\varepsilon^\EF_i(N_a)<\varepsilon^\EF_i(M)+1$.
\label{hyoka}
\item $\END_{\MH^\EF_{n-1}}(f^\EF_iM)\EISO \END_{\MH^\EF_{n-1}}(\Tf^\EF_iM)^{\oplus\varphi^\lambda_i(M)}$ as vector superspaces.
\label{end_dim2}
\end{enumerate}
\label{root_op2f}
\end{Lem}


\begin{Cor}
For any $M\in\IRR(\SMOD{\MH^\EF_n})$ and $i\in I_q$, we have 
$(e^\EF_i)^{\varepsilon^\EF_i(M)+1}[M]=(f^\EF_i)^{\varphi^\EF_i(M)+1}[M]=0$ in $K(\EF)$.
\label{nilp1}
\end{Cor}

\begin{proof}
$(e^\EF_i)^{\varepsilon^\EF_i(M)+1}[M]=0$ follows from Theorem \ref{root_op} (\ref{root_op1}).
To prove $(f^\EF_i)^{\varphi^\EF_i(M)+1}[M]=0$, it is enough to show that $(f^\EF_i)^{m}[M]\ne 0$
implies $(\Tf^\EF_i)^{m}M\ne 0$ for any $m\geq 0$.
By the definition, $(f^\EF_i)^{m}[M]\ne 0$ is equivalent to $[(\IND^\EF_i)^{m} M]\ne 0$. 
By Corollary \ref{ind_res2} (\ref{proj_proj}), we have 
\begin{align}
\begin{split}
\HOM_{\MH^\EF_{n+m}}((\IND^\EF_i)^{m} M,N)
&\cong
\HOM_{\MH^\EF_{n}}(M,(\RES^\EF_i)^{m}N) \\
&=
\HOM_{\MH_{n}}(\INFL^\EF M,\RES^{_{\MH_{n,m}}}_{_{\MH_{n}}}\Delta_{i^m}\INFL^\EF N)
\end{split}
\label{hom_temp}
\end{align}
for any $N\in\SMOD{\MH^\EF_{n+m}}$. Since $(\IND^\EF_i)^{m} M\ne 0$, 
there exists an $N\in \IRR(\SMOD{\MH^\EF_{n+m}})$ such that (\ref{hom_temp}) is non-zero. 
Take any irreducible sub $\MH_{n}$-supermodule $X\cong \INFL^\EF M$ of $\RES^{_{\MH_{n,m}}}_{_{\MH_{n}}}\Delta_{i^m}\INFL^\EF N$
and consider $\MH_{n,m}$-supermodule $X'\DEF\MH'_mX$ where $\MH'_m$ stands for a subsuperalgebra in $\MH_{n+m}$ 
generated by $\{X^{\pm 1}_k,C_k,T_l\mid n<k\leq n+m, n<l<n+m\}$ isomorphic to $\MH_m$. 
Then $\CH_{(n,m)}X'=c\cdot[X\MARU L(i^m)]$ for some $c\in\Z_{\geq 1}$ by Theorem \ref{kato}.
Comparing with $\SOC\Delta_{i^m}\INFL^\EF N\cong (\Te_i^m\INFL^\EF N)\MARU L(i^m)$ by 
Theorem \ref{Kato_irr_thm} (\ref{type_thm}) (see also ~\cite[Lemma 5.9.(i)]{BK}),
we see $(\INFL^\EF M\cong )X\cong \Te_i^m\INFL^\EF N$ which implies $(\Tf^\EF_i)^mM\cong N\ne 0$.
\end{proof}

As proved in ~\cite[Lemma 7.14]{BK}, $[\RES^\EF_i\IND^\EF_jM]-[\IND^\EF_j\RES^\EF_i M]$ is a multiple of $[M]$
for any $M\in\IRR(\SMOD{\MH^\EF_n})$. By Theorem \ref{root_op} (\ref{root_op1}) and 
Lemma \ref{root_op2f} (\ref{hyoka}), it implies the following.

\begin{Cor}
For any $M\in\IRR(\SMOD{\MH^\EF_n})$ and $i, j\in I_q$, we have 
$e^\EF_i(f^\EF_j [M])-f^\EF_j(e^\EF_i[M]) = \delta_{i,j}(\varphi^\EF_i(M)-\varepsilon^\EF_i(M))\cdot[M]$ in $K(\EF)$.
\label{nilp2}
\end{Cor}

By Schur's lemma, Theorem \ref{ind_res} (\ref{ind_res_res_ind}), 
Theorem \ref{root_op} (\ref{root_op3}), Lemma \ref{root_op2f} (\ref{limit_pt}) and Lemma \ref{root_op2f} (\ref{end_dim2}), 
we have the following. See also ~\cite[Lemma 6.20]{BK}.
\begin{Cor}
For any $M\in\IRR(\SMOD{\MH^\EF_n})$, we have
\begin{align*}
\sum_{i\in I_q}(2-\delta_{b_+(i),b_-(i)})(\varphi^\EF_i(M)-\varepsilon^\EF_i(M))=d.
\end{align*}
\label{nilp3}
\end{Cor}

\subsection{Left $K(\infty)^*$-module structure on $K(\EF)$}
\label{mod_str_cyc}
Clearly, $\INFL^\EF$ induces an injection $K(\EF)\hookrightarrow K(\infty)$ and a map
$\Delta^\EF:K(\EF)\to K(\EF)\otimes K(\infty)$ with the following commutative diagram
\begin{align*}
\xymatrix{
K(\infty) \ar[r]^{\!\!\!\!\!\!\!\!\!\!\!\!\!\!\!\!\Delta} & K(\infty)\otimes K(\infty) \\
K(\EF) \ar[r]^{\!\!\!\!\!\!\!\!\!\!\!\!\!\!\Delta^\EF} \ar@{^{(}->}[u]^{\INFL^\EF} & K(\EF)\otimes K(\infty)  \ar@{^{(}->}[u]_{\INFL^\EF\otimes \ID_{K(\infty)}} .
}
\end{align*}

Thus, $K(\EF)$ is a subcomodule of the right regular $K(\infty)$-comodule.
It implies that $K(\EF)$ is a $K(\infty)^*$-submodule of a left $K(\infty)^*$-module 
$K(\infty)$ in \S\ref{module_str_inf} where an operator $(e^\EF_i)^{(r)}$ acts 
as $\delta_{L(i^r)}$ by Lemma \ref{divided_root_op} for $i\in I_q$ and $r\geq 1$.

\subsection{Injectivity of the Cartan map}
The purpose of this subsection is to show 
the injectivity of the Cartan map $\omega_{\MH^\EF_n}$ of $\MH^\EF_n$~\cite[Theorem 7.10]{BK}.
It is essentially the same as ~\cite[\S7-c]{BK} but arguments are slightly different because
we don't define divided power operators $e^{(r)}_i, (e^{R}_i)^{(r)}$ and $(f^{R}_i)^{(r)}$ in a module-theoretic 
way as ~\cite[\S6-c]{BK}. 

We first recall the following formula~\cite[Lemma 7.6]{BK} which follows from
the definitions that $e^\EF_i$ and $f^\EF_i$ are suitable summands of $\RES^\EF_i$ and $\IND^\EF_i$ respectively.

\begin{Lem}
For any $x\in \KKK(\PROJ \MH^\EF_n)$ and $y_{\pm}\in \KKK(\SMOD{\MH^\EF_{n\pm 1}})$, we have
\begin{align*}
\langle e^\EF_ix, y_- \rangle_{\MH^\EF_{n-1}} = \langle x, f^\EF_iy_- \rangle_{\MH^\EF_n},\quad
\langle f^\EF_ix, y_+ \rangle_{\MH^\EF_{n+1}} = \langle x, e^\EF_iy_+ \rangle_{\MH^\EF_n}.
\end{align*}
\label{adj_form}
\end{Lem}

Since $(e^\EF_i)^{(r)}$ is a well-defined operator on $K(\EF)$, we have the following.
See also ~\cite[Corollary 7.7]{BK}.

\begin{Cor}
$(f^\EF_i)^{(r)}$ is a well-defined operator on $K(\EF)^*$ 
for any $i\in I_q$ and $r\geq 1$. More precisely, if
\begin{align*}
(e^\EF_i)^{(r)}[M]=\sum_{N\in \IRR(\SMOD{\MH^\EF_{n-r}})}a_{M,N}[N],\quad
(f^\EF_i)^{(r)}[M]=\sum_{N\in \IRR(\SMOD{\MH^\EF_{n+r}})}b_{M,N}[N]
\end{align*}
in $\KKK(\SMOD{\MH^\EF_{n-r}})$ and $\Q\otimes \KKK(\SMOD{\MH^\EF_{n+r}})$ respectively, then we have
\begin{align*}
(f^\EF_i)^{(r)}[P_N]=\sum_{M\in \IRR(\SMOD{\MH^\EF_{n+r}})}a_{M,N}[P_M],\quad
(e^\EF_i)^{(r)}[P_N]=\sum_{M\in \IRR(\SMOD{\MH^\EF_{n-r}})}b_{M,N}[P_M]
\end{align*}
in $\KKK(\PROJ\MH^\EF_{n+r})$ and $\Q\otimes \KKK(\PROJ\MH^\EF_{n-r})$ respectively.
\label{adj_formula}
\end{Cor}

\begin{Lem}
Let $M\in\IRR(\SMOD{\MH^\EF_n})$ and $i\in I_q$. For $m\leq \varepsilon\DEF\varepsilon^\EF_i(M)$, we have
\begin{align}
(e^\EF_i)^{m}[P_M]=\sum_{\substack{L\in\IRR(\SMOD{\MH^\EF_{n-m}}) \\ \varepsilon^\EF_i(L)\geq \varepsilon-m}}b_L[P_L]
\label{kajo}
\end{align}
in $\KKK(\PROJ\MH^\EF_{n-m})$. Moreover, in case $m=\varepsilon$, we have
\begin{align*}
(e^\EF_i)^{\varepsilon}[P_M]
=\varepsilon!{\varepsilon+\varphi^\EF_i(M) \choose \varepsilon}[P_{(\Te^\EF_i)^\varepsilon M}]
+\sum_{\substack{L\in\IRR(\SMOD{\MH^\EF_{n-\varepsilon}}) \\ \varepsilon^\EF_i(L)>0}}b_L[P_L].
\end{align*}
\label{hyoka2}
\end{Lem}

\begin{proof}
By Corollary \ref{adj_formula}, $b_L$ is 
the coefficient of $[M]$ in $(f^\EF_i)^{m}[L]$ in $\KKK(\SMOD{\MH^\EF_{n}})$.
Note by Lemma \ref{root_op2f} (\ref{hyoka}), we have
\begin{align*}
(f^\EF_i)^{m}[L]\in\sum_{\substack{N\in\IRR(\SMOD{\MH^\EF_{n}}) \\ \varepsilon^\EF_i(N)\leq m+\varepsilon^\EF_i(L)}}\Z_{\geq 0}[N].
\end{align*}
This implies $\varepsilon\leq m+\varepsilon^\EF_i(L)$ if $b_L\ne 0$ and completes the proof of (\ref{kajo}).

Suppose $b_L\ne 0$ and $\varepsilon^\EF_i(L)=0$. Again, by Lemma \ref{root_op2f} (\ref{hyoka}),
we have $(\Tf^\EF_i)^\varepsilon L \cong M$ and $b_L=\varepsilon!{\varphi^\EF_i(L) \choose \varepsilon}$.
Thus, we have $L\cong (\Te^\EF_i)^\varepsilon M$ and $b_L=\varepsilon!{\varepsilon+\varphi^\EF_i(M) \choose \varepsilon}$.
\end{proof}

\begin{Thm}
$\omega_{\MH^\EF_n}:\KKK(\PROJ\MH^\EF_n)\to\KKK(\SMOD{\MH^\EF_n})$ is injective for all $n\geq 0$.
\label{two_lattice}
\end{Thm}

\begin{proof}
We prove by induction on $n$. The case $n=0$ is clear.

Suppose $n>0$ and $\omega_{\MH^\EF_{n'}}$ is injective for all smaller $n'<n$.
We show that if 
\begin{align}
\omega_{\MH^\EF_n}(\sum_{M\in\IRR(\SMOD{\MH^\EF_n})} a_M[P_M])=0
\label{asum}
\end{align}
for $a_M\in \Z$,
then we have $a_M=0$ for all $M\in\IRR(\SMOD{\MH^\EF_n})$.
To prove it, it is enough to show that for each $i\in I_q$ 
we have $a_M=0$ for all $M\in\IRR(\SMOD{\MH^\EF_n})$ with $\varepsilon^\EF_i(M)>0$. 
This is because 
there exists some $i\in I_q$ such that $\varepsilon^\EF_i(M)>0$ for any $M\in\IRR(\SMOD{\MH^\EF_n})$ if $n>0$.

For each $i\in I_q$, we prove it by induction on $\varepsilon^\EF_i(M)>0$.
Suppose that for a given $M$ with $\varepsilon\DEF\varepsilon^\EF_i(M)> 0$ 
we have $a_N=0$ for all $N$ 
with $0<\varepsilon^\EF_i(N)<\varepsilon$.
Apply $(e^\EF_i)^\varepsilon$ to (\ref{asum}), we have
\begin{align*}
0 = \sum_{\substack{L\in\IRR(\SMOD{\MH^\EF_n}) \\ \varepsilon^\EF_i(L)=\varepsilon}}
\varepsilon!{\varepsilon+\varphi^\EF_i(L) \choose \varepsilon}a_L\omega_{\MH^\EF_{n-\varepsilon}}([P_{(\Te^\EF_i)^\varepsilon L}])
+\omega_{\MH^\EF_{n-\varepsilon}}(X)
\end{align*}
where
$
X\in\sum_{\textrm{$L'\in\IRR(\SMOD{\MH^\EF_{n-\varepsilon}})$ with $\varepsilon^\EF_i(L')>0$}}\Z[P_{L'}]
$ by Lemma \ref{commutativity_proj} and Lemma \ref{hyoka2}.
By induction hypothesis, we have $a_M=0$.
\end{proof}

\subsection{Symmetric non-degenerate bilinear form on $K(\EF)_\Q$}
\label{symmetric_form}

By Theorem \ref{two_lattice}, $\bigoplus_{n\geq 0}\KKK(\PROJ\MH^\EF_n)\cong K(\EF)^*\subseteq K(\EF)$ 
are two integral lattices of $K(\EF)_\Q\DEF \Q\otimes K(\EF)$.
Thus, by tensoring $\Q$, $\bigoplus_{n\geq 0}\langle,\rangle_{\MH^\EF_n}:K(\EF)^*\times K(\EF)\to\Z$ 
induces a non-degenerate bilinear form on $K(\EF)_\Q$ which we denote by $\langle,\rangle_\EF$.

\begin{Lem}
Let $M\in\IRR(\SMOD{\MH^\EF_{n}})$ and $i\in I_q$. We have 
\begin{align*}
[P_M] = (f^\EF_i)^{(\varepsilon)}[P_{(\Te^\EF_i)^\varepsilon M}]-\sum_{\substack{L\in\IRR(\SMOD{\MH^\EF_{n}}) \\ \varepsilon^\EF_i(L)>\varepsilon}}a_L[P_L]
\end{align*}
\label{ind_formula}
for $\varepsilon=\varepsilon^\EF_i(M)$
in $\KKK(\PROJ\MH^\EF_{n})$.
\end{Lem}

\begin{proof}
Write $(f^\EF_i)^{(\varepsilon)}[P_{(\Te^\EF_i)^\varepsilon M}]=\sum_{L\in\IRR(\SMOD{\MH^\EF_{n}})}b_L[P_L]$ in $\KKK(\PROJ\MH^\EF_{n})$.
By Corollary \ref{adj_formula}, $b_L$ is 
the coefficient of $[(\Te^\EF_i)^\varepsilon M]$ of $(e^\EF_i)^{(\varepsilon)}[L]$ in $\KKK(\SMOD{\MH^\EF_{n-\varepsilon}})$.
Thus, $b_L\ne 0$ implies $\varepsilon^\EF_i(L)\geq\varepsilon$. 
Finally,
suppose $b_L\ne 0$ and $\varepsilon^\EF_i(L)=\varepsilon$.
By Theorem \ref{root_op} (\ref{root_op1}), we have $b_L=1$ and $(\Te^\EF_i)^\varepsilon L\cong (\Te^\EF_i)^\varepsilon M$, i.e.,
$L\cong M$.
\end{proof}

A repeated use of Lemma \ref{ind_formula} implies the following~\cite[Theorem 7.9]{BK}.

\begin{Thm}
We have $\bigoplus_{n\geq 0}\KKK(\PROJ \MH^\EF_n)=U^-_\Z[\TRIVREP_\EF]$ where $\TRIVREP_\EF$ is the trivial 
supermodule of $\MH^\EF_0=F$.
\label{proj_lattice}
\end{Thm}

\begin{proof}
We prove $[P_M]\in U^-_\Z[\TRIVREP_\EF]$ for all $M\in B(\EF)$.
Suppose for a contradiction an existence of $M\in \IRR(\SMOD{\MH^\EF_n})$ 
such that $[P_M]\not\in U^-_\Z[\TRIVREP_\EF]$. We take such an $M$ with minimum $n$.
Since $n>0$, there exists an $i\in I_q$ with $\varepsilon^\EF_i(M)>0$.
We take $N$ with maximum $\varepsilon^\EF_i(N)(\geq \varepsilon^\EF_i(M)>0)$ 
in $\{N\in\IRR(\SMOD{\MH^\EF_n})\mid[P_N]\not\in U^-_\Z[\TRIVREP_\EF]\}(\ne\emptyset)$.
However, $[P_N]\in U^-_\Z[\TRIVREP_\EF]$ by a choice of $N$ and Lemma \ref{ind_formula}, a contradiction.
\end{proof}

Using Lemma \ref{adj_form} inductively along with 
$\KKK(\SMOD{\MH^\EF_{n+1}})_\Q=\sum_{i\in I_q}f^\EF_i\KKK(\SMOD{\MH^\EF_{n}})_\Q$ by Theorem \ref{proj_lattice},
we get the following result~\cite[Theorem 7.11]{BK}.

\begin{Cor}
The non-degenerate bilinear form $\langle,\rangle_\EF$ on $K(\EF)_\Q$ is symmetric.
\label{sym_bil}
\end{Cor}

\section{Character calculations}
\label{keisan}

The purpose of this section is to give preparatory character calculations 
concerning behavior of representations of low rank affine Hecke-Clifford superalgebras $\MH_2,\MH_3$ and $\MH_4$
for \S\ref{onaji_lemma}.
Since they are responsible for the appearance of Lie theory of type $D^{(2)}_l$ and
omitted in ~\cite{BK}, we give detailed and self-contained calculations.

\subsection{Preparations}

We note that if a given $M\in\IRR(\REP\MH_n)$ has a formal character
of the form $\CH M=c\cdot [L(i_i)\MARU\cdots\MARU L(i_n)]$ for some $c\in\Z_{\geq 1}$ then
$M\cong L(i_1,\cdots,i_n)$ by Corollary \ref{label_of_irr}. We also 
recall the Shuffle lemma~\cite[Lemma 4.11]{BK} to compute the formal characters.

\begin{Lem}
For $M\in\IRR(\REP\MH_m)$ and $N\in\IRR(\REP\MH_n)$ with
$\CH M=\sum_{\BII\in I_q^m}a_{\BII}[L(i_1)\MARU\cdots\MARU L(i_m)]$ and
$\CH N=\sum_{\BJJ\in I_q^n}b_{\BJJ}[L(j_1)\MARU\cdots\MARU L(j_n)]$,
we have
\begin{align*}
\CH \IND_{\MH_{m,n}}^{\MH_{m+n}}M\MARU N
=\sum_{\substack{\BII\in I_q^m\\ \BJJ\in I_q^n}}
a_{\BII}b_{\BJJ}(\sum_{\BKK\in I_q^{m+n}}[L(k_1)\MARU\cdots\MARU L(k_{m+n})]).
\end{align*}
Here $\BKK\in I_q^{m+n}$ runs satisfying the following condition:
there exist $1\leq u_1<\cdots<u_m\leq m+n$ and $1\leq v_1<\cdots<v_n\leq m+n$ such that
$(k_{u_1},\cdots,k_{u_m})=(i_1,\cdots,i_m), (k_{v_1},\cdots,k_{v_n})=(j_1,\cdots,j_n)$
and
$\{ u_1,\cdots,u_m \}\sqcup\{ v_1,\cdots,v_n \}=\{1,\cdots,m+n\}$.
\label{char_lemma}
\end{Lem}

We also need the following~\cite[Lemma 4.3]{BK} which follows by direct calculation.

\begin{Lem}
Suppose we are given $a,b\in F^\times$ with $a+a^{-1}=q(i)$ and $b+b^{-1}=q(j)$ for some $i,j\in I_q$.
If $|i-j|\leq 1$, then the following vanishes.
\begin{align*}
a^{-2}(ab-1)^2(ab^{-1}-1)^2(a^{-2}(ab-1)^2(ab^{-1}-1)^2-\xi^2a^{-1}b^{-1}(ab-1)^2-\xi^2a^{-1}b(ab^{-1}-1)^2).
\end{align*}
\label{tech}
\end{Lem}

\begin{Cor}
For any $i,j\in\Z$ with $|i-j|=1$ and $q(j)\ne q(i)$, we have 
\begin{align*}
\frac{\xi^2}{(q(j)-q(i))^2}(q(i)q(j)-4)=1.
\end{align*}
\label{fund_cor}
\end{Cor}

\begin{proof}
We take $a$ and $b$ to satisfy $a+a^{-1}=q(i)$ and $b+b^{-1}=q(j)$. We have
\begin{align*}
a^{-2}(ab-1)^2(ab^{-1}-1)^2-\xi^2(a^{-1}b^{-1}(ab-1)^2+a^{-1}b(ab^{-1}-1)^2)=0.
\end{align*}
by Lemma \ref{tech} and $q(i)\ne q(j)$.
Direct calculation shows that the left hand side is equal to 
$(q(i)-q(j))^2-\xi^2(q(i)q(j)-4)$.
\end{proof}

In the rest of this section, for each $i\in I_q$ we write the basis elements $w_1$ and $w'_1$ of $L(i)(=L^+_1(i))$
in Definition \ref{cover} as $v^i_{\BARR{0}}$ and $v^i_{\BARR{1}}$ respectively.
Recall that the irreducible $\MH_1$-supermodule $L(i)=Fv^{i}_{\BARR{0}}\oplus Fv^{i}_{\BARR{1}}$ 
is given by the grading $L(i)_{j}=Fv^i_{j}$ for $j\in\FF$ and
the following action.
\begin{align*}
X_1^\pm v^i_{\BARR{0}}=b_{\pm}(i)v^i_{\BARR{0}},\quad
X_1^\pm v^i_{\BARR{1}}=b_{\mp}(i)v^i_{\BARR{1}},\quad
C_1v^i_{\BARR{0}}=v^i_{\BARR{1}},\quad
C_1v^i_{\BARR{1}}=v^i_{\BARR{0}}.
\end{align*}

\subsection{On the block $[(i,j)]$ with $|i-j|=1$}

\begin{Lem}
For any $i,j\in\Z$ such that
\begin{align*}
|i-j|=1,\quad q(j)\ne q(i),\quad (\TYPE L(i),\TYPE L(j))\ne(\TQ,\TQ),
\end{align*}
we define $\MH_2$-supermodule $M$ and $\MA_2$-supermodule $N$ as follows.
\begin{align*}
M\DEF\IND_{\MH_{1,1}}^{\MH_2}L(j)\otimes L(i),\quad
N\DEF(X_2+X_2^{-1}-q(i))M\subseteq\RES^{\MH_2}_{\MH_{1,1}}M.
\end{align*}
Then the following two statements hold.
\begin{enumerate}
\item $N$ is $T_1$-invariant, i.e., $N$ is an $\MH_2$-supermodule.
\item $\CH N=[L(i)\otimes L(j)]$.
\end{enumerate}
\label{step1}
\end{Lem}

\begin{proof}
Note that we have $\CH_{1,1} N=[L(i)\otimes L(j)]$ 
because $0\subsetneq N\subsetneq M$ and $\CH M = [L(i)\otimes L(j)]+[L(j)\otimes L(i)]$ by 
Lemma \ref{char_lemma} and 
$\CH\HEAD(M)=\CH L(ji)$ contains a term $[L(j)\otimes L(i)]$ by Corollary \ref{label_of_irr}.
Thus, it is enough to show that $T_1N\subseteq N$.

By (\ref{non_triv_eq1}) and (\ref{non_triv_eq2}), we have
\begin{align*}
(X_2+X_2^{-1}-q(i))T_1=T_1(X_1+X_1^{-1}-q(i))+\xi(X_2+C_1C_2X_1-X_1^{-1}-X_2^{-1}C_1C_2).
\end{align*}
From this, we see that the following $X$ and $Y$ form a basis of $N_{\BARR{0}}$.
\begin{align*}
X &\DEF (X_2+X_2^{-1}-q(i))T_1\otimes v^j_{\BARR{0}}\otimes v^i_{\BARR{0}} \\
&=
(q(j)-q(i))T_1\otimes v^j_{\BARR{0}}\otimes v^i_{\BARR{0}}+\xi((b_+(i) - b_-(j))1\otimes v^j_{\BARR{0}}\otimes v^i_{\BARR{0}}
-(b_+(i)-b_+(j))1\otimes v^j_{\BARR{1}}\otimes v^i_{\BARR{1}}), \\
Y &\DEF (X_2+X_2^{-1}-q(i))T_1\otimes v^j_{\BARR{1}}\otimes v^i_{\BARR{1}} \\
&=
(q(j)-q(i))T_1\otimes v^j_{\BARR{1}}\otimes v^i_{\BARR{1}}+\xi((b_-(i) - b_-(j))1\otimes v^j_{\BARR{0}}\otimes v^i_{\BARR{0}}
+(b_-(i)-b_+(j))1\otimes v^j_{\BARR{1}}\otimes v^i_{\BARR{1}}).
\end{align*}

To show $T_1N\subseteq N$, it is enough to show $T_1N_{\BARR{0}}\subseteq N_{\BARR{0}}$. 
For this purpose, it is enough to show the following equalities which follows from Corollary \ref{fund_cor}.
\begin{align*}
T_1X &= \xi(1+\frac{b_+(i)-b_-(j)}{q(j)-q(i)})X-\xi\frac{b_+(i)-b_+(j)}{q(j)-q(i)}Y, \\
T_1Y &= \xi\frac{b_-(i)-b_-(j)}{q(j)-q(i)}X+\xi(1+\frac{b_-(i)-b_+(j)}{q(j)-q(i)})Y.
\end{align*}

\end{proof}

\begin{Cor}
For any $i,j\in\Z$ such that
\begin{align*}
|i-j|=1,\quad q(j)\ne q(i),\quad (\TYPE L(i),\TYPE L(j))\ne(\TQ,\TQ).
\end{align*}
We have the following descriptions of $L(ij)$.
\begin{enumerate}
\item $\CH L(ij)=[L(i)\otimes L(j)]$.
\item There exists a basis $\{X,Y\}$ of $L(ij)_{\BARR{0}}$ such that
the matrix representations of $L(ij)$ with respect to the basis $\{X,Y,C_1X,C_1Y\}$ is as follows.
\begin{align*}
X_1^{\pm1} &: 
\begin{pmatrix}
b_\pm(i) & 0 & 0 & 0 \\
0 & b_\mp(i) & 0 & 0 \\
0 & 0 & b_\pm(i) & 0 \\
0 & 0 & 0 & b_\mp(i) \\
\end{pmatrix},\quad
X_2^{\pm1} : 
\begin{pmatrix}
b_\pm(j) & 0 & 0 & 0 \\
0 & b_\mp(j) & 0 & 0 \\
0 & 0 & b_\pm(j) & 0 \\
0 & 0 & 0 & b_\mp(j) \\
\end{pmatrix},\\
C_1 &: 
\begin{pmatrix}
0 & 0 & 1 & 0 \\
0 & 0 & 0 & 1 \\
1 & 0 & 0 & 0 \\
0 & 1 & 0 & 0 \\
\end{pmatrix},\quad
C_2 : 
\begin{pmatrix}
0 & 0 & 0 & -1 \\
0 & 0 & 1 & 0 \\
0 & 1 & 0 & 0 \\
-1 & 0 & 0 & 0 \\
\end{pmatrix},\\
T_1 &: 
\frac{\xi}{q(j)-q(i)}
\begin{pmatrix}
b_+(j)-b_-(i) & b_-(i)-b_-(j) & 0 & 0 \\
b_+(j)-b_+(i) & b_-(j)-b_+(i) & 0 & 0 \\
0 & 0 & b_+(j)-b_+(i) & b_-(j)-b_+(i) \\
0 & 0 & b_-(i)-b_+(j) & b_-(j)-b_-(i) \\
\end{pmatrix}.
\end{align*}
\end{enumerate}
\label{step2}
\end{Cor}

\subsection{On the block $[(i,i,j)]$ with $|i-j|=1$}

\begin{Lem}
For any $i,j\in\Z$ such that
\begin{align*}
|i-j|=1,\quad q(j)\ne q(i),\quad (\TYPE L(i),\TYPE L(j))=(\TM,\TM).
\end{align*}
We define $\MH_3$-supermodule $M$ and $\MH_{2,1}$-supermodule $N$ as follows.
\begin{align*}
M\DEF\IND_{\MH_{2,1}}^{\MH_3}L(ij)\otimes L(i),\quad
N\DEF(X_3+X_3^{-1}-q(i))M\subseteq\RES^{\MH_3}_{\MH_{2,1}}M.
\end{align*}
If $q(i)q(j)+q(j)^2-8\ne 0$, then we have $T_2N\not\subseteq N$ and $M$ is irreducible.
\label{tochuu}
\end{Lem}

\begin{proof}
Since $\CH\HEAD M=L(iji)$ contains a term $[L(i)\otimes L(j)\otimes L(i)]$ by
Corollary \ref{label_of_irr} and $\CH M=[L(i)\otimes L(j)\otimes L(i)] + 2[L(i)^{\otimes 2}\otimes L(j)]$
by Lemma \ref{char_lemma}, if $M$ is reducible then $M$ has a unique irreducible submodule $M'$
with $\RES^{\MH_3}_{\MH_{2,1}}M'\cong L(i^2)\otimes L(j)$ by Theorem \ref{kato}. 
Thus, if $M$ is reducible then $\RES^{\MH_3}_{\MH_{2,1}}M'=N$.
It implies that if $T_2N\not\subseteq N$ then $M$ is irreducible.

In the rest of the proof, we show that $T_2N\not\subseteq N$ if $q(i)q(j)+q(j)^2-8\ne 0$.
We take a basis $(\alpha_1,\alpha_2,\alpha_3,\alpha_4)\DEF(X,Y,C_1X,C_1Y)$ of $L(ij)$ in Corollary \ref{step2}.
Then a basis of $M$ is given by
\begin{align*}
\{X_{\beta,k,l}\DEF\beta\otimes\alpha_k\otimes v^i_{l}\mid \beta\in\{1,T_2,T_1T_2\},k\in\{1,2,3,4\},l\in\FF\}
\end{align*}
and a basis of $N_{\BARR{0}}$ is given by $\{Y_k, Z_k\mid 1\leq k\leq 4\}$ where
\begin{align*}
Y_k\DEF (X_3+X_3^{-1}-q(i))X_{T_2,k,f(k)},\quad Z_k\DEF (X_3+X_3^{-1}-q(i))X_{T_1T_2,k,f(k)}(=T_1Y_k)
\end{align*}
for $k=1,2,3,4$ and $f(1)=f(2)=\BARR{0}$ and $f(3)=f(4)=\BARR{1}$.
More explicitly,
\allowdisplaybreaks{
\begin{align*}
Y_1
&=
(q(j)-q(i))T_2\otimes\alpha_1\otimes v^i_{\BARR{0}}
+\xi((b_+(i)-b_-(j))1\otimes\alpha_1\otimes v^i_{\BARR{0}}+(b_+(i)-b_+(j))1\otimes\alpha_4\otimes v^i_{\BARR{1}}), \\
Y_2
&=
(q(j)-q(i))T_2\otimes\alpha_2\otimes v^i_{\BARR{0}} 
+\xi((b_+(i)-b_+(j))1\otimes\alpha_2\otimes v^i_{\BARR{0}}+(b_-(j)-b_+(i))1\otimes\alpha_3\otimes v^i_{\BARR{1}}), \\
Y_3
&=
(q(j)-q(i))T_2\otimes\alpha_3\otimes v^i_{\BARR{1}} 
+\xi((b_-(i)-b_-(j))1\otimes\alpha_3\otimes v^i_{\BARR{1}}+(b_-(i)-b_+(j))1\otimes\alpha_2\otimes v^i_{\BARR{0}}), \\
Y_4
&=
(q(j)-q(i))T_2\otimes\alpha_4\otimes v^i_{\BARR{1}} 
+\xi((b_-(i)-b_+(j))1\otimes\alpha_4\otimes v^i_{\BARR{1}}+(b_-(j)-b_-(i))1\otimes\alpha_1\otimes v^i_{\BARR{0}}), \\
Z_1 
&= (q(j)-q(i))T_1T_2\otimes\alpha_1\otimes v^i_{\BARR{0}} 
+\frac{\xi^2}{q(j)-q(i)}((b_+(i)-b_-(j))(b_+(j)-b_-(i))1\otimes\alpha_1\otimes v^i_{\BARR{0}} \\
&{\quad}
+(b_+(i)-b_-(j))(b_+(j)-b_+(i))1\otimes\alpha_2\otimes v^i_{\BARR{0}}
+(b_+(i)-b_+(j))(b_-(j)-b_+(i))1\otimes\alpha_3\otimes v^i_{\BARR{1}} \\
&{\quad}
+(b_+(i)-b_+(j))(b_-(j)-b_-(i))1\otimes\alpha_4\otimes v^i_{\BARR{1}}), \\
Z_2 
&= (q(j)-q(i))T_1T_2\otimes\alpha_2\otimes v^i_{\BARR{0}} 
+\frac{\xi^2}{q(j)-q(i)}((b_+(i)-b_+(j))(b_-(i)-b_-(j))1\otimes\alpha_1\otimes v^i_{\BARR{0}} \\
&{\quad}
+(b_+(i)-b_+(j))(b_-(j)-b_+(i))1\otimes\alpha_2\otimes v^i_{\BARR{0}} 
+(b_-(j)-b_+(i))(b_+(j)-b_+(i))1\otimes\alpha_3\otimes v^i_{\BARR{1}} \\
&{\quad}
+(b_-(j)-b_+(i))(b_-(i)-b_+(j))1\otimes\alpha_4\otimes v^i_{\BARR{1}}).
\end{align*}}

It is enough to show  $T_2Z_1\not\in N_{\BARR{0}}$ to prove $T_2N_{\BARR{0}}\not\subseteq N_{\BARR{0}}$. 
Note that we have
\begin{align*}
T_2Z_1=\xi((b_+(j)-b_-(i))T_1T_2\otimes\alpha_1\otimes v^i_{\BARR{0}}
+(b_+(j)-b_+(i))T_1T_2\otimes\alpha_2\otimes v^i_{\BARR{0}})+\Delta
\end{align*}
for a suitable $\Delta\in\SPAN\{X_{T_2,k,l}\mid1\leq k\leq 4, l\in\FF\}$.
Thus, if $T_2Z_1\in N_{\BARR{0}}$, then we must have
\allowdisplaybreaks{
\begin{align*}
T_2Z_1
&=
\xi\frac{b_+(j)-b_-(i)}{q(j)-q(i)}Z_1+\xi\frac{b_+(j)-b_+(i)}{q(j)-q(i)}Z_2 \\
&{\quad}
+\frac{\xi^2}{(q(j)-q(i))^2}((b_+(i)-b_-(j))(b_+(j)-b_-(i))Y_1+(b_+(i)-b_-(j))(b_+(j)-b_+(i))Y_2 \\
&{\quad\quad\quad\quad\quad\quad\quad\quad} 
+(b_+(i)-b_+(j))(b_-(j)-b_+(i))Y_3+(b_+(i)-b_+(j))(b_-(j)-b_-(i))Y_4).
\end{align*}}
Especially, the coefficient of $1\otimes\alpha_1\otimes v^i_{\BARR{0}}$ of the right hand side must be 0.
It gives us 
\begin{align*}
\frac{\xi^3}{(q(j)-q(i))^2}(b_+(i)-b_-(i))(q(i)q(j)+q(j)^2-8)=0.
\end{align*}
Thus, we have $T_2Z_1\not\in N_{\BARR{0}}$ if $q(i)q(j)+q(j)^2-8\ne0$.
\end{proof}

\begin{Cor}
Assume $q$ be a primitive $4l$-th root of unity for $l\geq 3$ and assume $i,j\in\Z$ satisfy 
\begin{align*}
|i-j|=1,\quad q(j)\ne q(i),\quad (\TYPE L(i),\TYPE L(j))=(\TM,\TM).
\end{align*}
Then we have the following descriptions.
\begin{enumerate}
\item $L(iji)\cong L(iij)\cong \IND_{2,1}^{3}L(ij)\otimes L(i)$.
\item $\CH L(iji)=\CH L(iij)=2[L(i)^{\otimes2} \otimes L(j)]+[L(i)\otimes L(j)\otimes L(i)]$.
\item $\CH L(jii)=2[L(j)\otimes L(i)^{\otimes2}]+[L(i)\otimes L(j)\otimes L(i)]$.
\end{enumerate}
\label{tochuu_cor}
\end{Cor}

\begin{proof}
$q(i)q(j)+q(j)^2-8=0$ is equivalent to $q^{4i+3\pm 3}=1$ or $q^{4i+1\pm 3}=1$ by
\begin{align*}
{} &{\quad}
(q^{2j+1}+q^{-2(j+1)})^2+(q^{2i+1}+q^{-2(i+1)})(q^{2j+1}+q^{-2(j+1)})-2(q+q^{-1})^2 \\
&=
(q^{2(i\pm 1)+1}+q^{-2((i\pm 1)+1)})^2+(q^{2i+1}+q^{-2(i+1)})(q^{2(i\pm 1)+1}+q^{-2((i\pm 1)+1)})-2(q+q^{-1})^2 \\
&=
(q+q^{-1})(q^{2i+1.5\pm1.5}-q^{-(2i+1.5\pm1.5)})(q^{2i+0.5\pm1.5}-q^{-(2i+0.5\pm1.5)}).
\end{align*}
Since $\TYPE L(i)=\TM$, we have $l\geq 3$ and $1\leq i\leq l-2$.
Thus we have $2\leq 4i-2<4i+6\leq 4l-2$ and we see that $q^{4i+3\pm 3}\ne1$ and $q^{4i+1\pm 3}\ne1$.

By Lemma \ref{tochuu}, $L(iji)\cong M\DEF\IND_{2,1}^{3}L(ij)\otimes L(i)$.
Thus, $\CH L(iji)=2[L(i)^{\otimes2} \otimes L(j)]+[L(i)\otimes L(j)\otimes L(i)]$ by Lemma \ref{char_lemma}.
It implies $\Delta_j M\ne 0$ and $\Te_j M\cong L(i^2)$ by Theorem \ref{kato}. Thus, we have $M\cong L(iij)$.

Finally, consider the irreducible supermodule $L(iij)^\sigma$. It belongs to the same block as $L(iij)\cong L(iji)$,
however it is non-isomorphic to $L(iij)\cong L(iji)$ in virtue of their formal characters.
Thus, we have $L(iij)^\sigma\cong L(jii)$. 
\end{proof}

\begin{Lem}
For any $i,j\in\Z$ such that
\begin{align*}
|i-j|=1,\quad q(j)\ne q(i),\quad (\TYPE L(i),\TYPE L(j))=(\TQ,\TM),
\end{align*}
we define $\MH_3$-supermodule $M$ and $\MH_{2,1}$-supermodule $N$ as follows.
\begin{align*}
M\DEF\IND_{\MH_{2,1}}^{\MH_3}L(ij)\MARU L(i),\quad
N\DEF(X_3+X_3^{-1}-q(i))M\subseteq\RES^{\MH_3}_{\MH_{2,1}}M.
\end{align*}
Then the following two statements hold.
\begin{enumerate}
\item $N$ is $T_2$-invariant, i.e., $N$ is an $\MH_3$-supermodule.
\item $\CH N=2[L(i)^{\MARU 2}\MARU L(j)]$ and $\CH M/N=[L(i)\MARU L(j)\MARU L(i)]$.
\end{enumerate}
\label{step3}
\end{Lem}

\begin{proof}
As in the first paragraph of the proof of Lemma \ref{tochuu}, if $N$ is $T_2$-invariant then
$\CH N=2[L(i)^{\MARU 2}\MARU L(j)]$ and $\CH M/N=[L(i)\MARU L(j)\MARU L(i)]$. 
Thus, it is enough we show that $N$ is $T_2$-invariant.

In the rest of the proof, we write $a$ instead of $b_+(i)=b_-(i)$ and 
take a basis $\{X,Y,C_1X,C_1Y\}$ of $L(ij)$ in Corollary \ref{step2}.
%

We can take a realization of $L(ij)\MARU L(i)$ as an $\MH_{2,1}$-submodule $W$ of $L(ij)\otimes L(i)$
given as follows because direct calculation shows that $W$ is $\MH_{2,1}$-invariant.
\begin{align*}
W\DEF W_{\BARR{0}}\oplus W_{\BARR{1}},\quad
W_{\BARR{0}}\DEF FX'\oplus FY',\quad W_{\BARR{1}}\DEF F(C_1X')\oplus F(C_1Y'), \\
X'\DEF X\otimes v^i_{\BARR{0}}+\sqrt{-1}(C_1X)\otimes v^i_{\BARR{1}},\quad
Y'\DEF Y\otimes v^i_{\BARR{0}}-\sqrt{-1}(C_1Y)\otimes v^i_{\BARR{1}}.
\end{align*} 
More precisely, we can check that 
the matrix representations with respect to the basis $\{X',Y',C_1X',C_1Y'\}$ is given as follows.
\begin{align*}
X_1^{\pm1} &: 
\begin{pmatrix}
a & 0 & 0 & 0 \\
0 & a & 0 & 0 \\
0 & 0 & a & 0 \\
0 & 0 & 0 & a \\
\end{pmatrix},\quad
X_2^{\pm1} : 
\begin{pmatrix}
b_\pm(j) & 0 & 0 & 0 \\
0 & b_\mp(j) & 0 & 0 \\
0 & 0 & b_\pm(j) & 0 \\
0 & 0 & 0 & b_\mp(j) \\
\end{pmatrix},\quad
X_3^{\pm1} : 
\begin{pmatrix}
a & 0 & 0 & 0 \\
0 & a & 0 & 0 \\
0 & 0 & a & 0 \\
0 & 0 & 0 & a \\
\end{pmatrix}, \\
C_1 &: 
\begin{pmatrix}
0 & 0 & 1 & 0 \\
0 & 0 & 0 & 1 \\
1 & 0 & 0 & 0 \\
0 & 1 & 0 & 0 \\
\end{pmatrix},\quad
C_2 : 
\begin{pmatrix}
0 & 0 & 0 & -1 \\
0 & 0 & 1 & 0 \\
0 & 1 & 0 & 0 \\
-1 & 0 & 0 & 0 \\
\end{pmatrix},\quad
C_3 : 
\begin{pmatrix}
0 & 0 & \sqrt{-1} & 0 \\
0 & 0 & 0 & -\sqrt{-1} \\
-\sqrt{-1} & 0 & 0 & 0 \\
0 & \sqrt{-1} & 0 & 0 \\
\end{pmatrix}, \\
T_1 &: 
\frac{\xi}{q(j)-q(i)}
\begin{pmatrix}
b_+(j)-a & a-b_-(j) & 0 & 0 \\
b_+(j)-a & b_-(j)-a & 0 & 0 \\
0 & 0 & b_+(j)-a & b_-(j)-a \\
0 & 0 & a-b_+(j) & b_-(j)-a \\
\end{pmatrix}.
\end{align*}

Hereafter, we put $(\alpha_1, \alpha_2, \alpha_3, \alpha_4)\DEF (X',Y',C_1X',C_1Y')$.
Then a basis of $M$ is given by
$\{X_{\beta,k}\DEF \beta\otimes\alpha_k\mid \beta\in\{1,T_2,T_1T_2\},k\in\{1,2,3,4\}\}$.
It is enough to show that $T_2N_{\BARR{0}}\subseteq N_{\BARR{0}}$. We can choose 
\begin{align*}
\{Y_k\DEF(X_3+X_3^{-1}-q(i))X_{T_2,k}, Y_{k+2}\DEF(X_3+X_3^{-1}-q(i))X_{T_1T_2,k}\mid 1\leq k\leq 2\}
\end{align*} 
as a basis of $N_{\BARR{0}}$. 
More explicitly, we have
\begin{align*}
Y_1
&=
(q(j)-q(i))T_2\otimes\alpha_1+\xi((a-b_-(j))1\otimes\alpha_1+\sqrt{-1}(a-b_+(j))1\otimes\alpha_2), \\
Y_2
&=
(q(j)-q(i))T_2\otimes\alpha_2+\xi(\sqrt{-1}(a-b_-(j))1\otimes\alpha_1+(a-b_+(j))1\otimes\alpha_2), \\
Y_3
&=
(q(j)-q(i))T_1T_2\otimes\alpha_1 \\
&\quad
+\frac{\xi^2}{q(j)-q(i)}(b_+(j)-a)(a-b_-(j))
((1-\sqrt{-1})1\otimes\alpha_1+(1+\sqrt{-1})1\otimes\alpha_2), \\
Y_4
&=
(q(j)-q(i))T_1T_2\otimes\alpha_2 \\
&\quad
+\frac{\xi^2}{q(j)-q(i)}(b_+(j)-a)(a-b_-(j))
((-1+\sqrt{-1})1\otimes\alpha_1+(1+\sqrt{-1})1\otimes\alpha_2).
\end{align*}
Now we can check the following relations using Corollary \ref{fund_cor}.
\allowdisplaybreaks{
\begin{align*}
T_2Y_1 
&=
\xi\frac{b_+(j)-a}{q(j)-q(i)}Y_1+\xi\frac{(a-b_+(j))\sqrt{-1}}{q(j)-q(i)}Y_2, \\
T_2Y_2 
&=
\xi\frac{(a-b_-(j))\sqrt{-1}}{q(j)-q(i)}Y_1+\xi\frac{b_-(j)-a}{q(j)-q(i)}Y_2, \\
T_2Y_3
&=
\frac{\xi(b_+(j)-a)}{q(j)-q(i)}(Y_3+Y_4)+\frac{\xi^2(b_+(j)-a)(a-b_-(j))}{(q(j)-q(i))^2}
((1-\sqrt{-1})Y_1+(1+\sqrt{-1})Y_2), \\
T_2Y_4
&=
\frac{\xi(a-b_-(j))}{q(j)-q(i)}(Y_3-Y_4)+\frac{\xi^2(b_+(j)-a)(a-b_-(j))}{(q(j)-q(i))^2}
((-1+\sqrt{-1})Y_1+(1+\sqrt{-1})Y_2).
\end{align*}}
\end{proof}

\begin{Cor}
For any $i,j\in\Z$ such that
\begin{align*}
|i-j|=1,\quad q(j)\ne q(i),\quad (\TYPE L(i),\TYPE L(j))=(\TQ,\TM),
\end{align*}
we have the following descriptions. Here we write $a=b_+(i)=b_-(i)$.
\begin{enumerate}
\item $\CH L(iij)=2[L(i)^{\MARU 2}\MARU L(j)]$ and $\CH L(iji)=[L(i)\MARU L(j)\MARU L(i)]$.
\item There exists a basis $\{Y_1,Y_2,Y_3,Y_4\}$ of $L(iij)_{\BARR{0}}$ such that
\allowdisplaybreaks{
\begin{align*}
Y_3 &= T_1Y_1,\quad Y_4=T_1Y_2, \\
X_3^{\pm1} Y_1&=b_\pm(j) Y_1,\quad
X_3^{\pm1} Y_2=b_\mp(j) Y_2,\quad
X_3^{\pm1} Y_3=b_\pm(j) Y_3,\quad
X_3^{\pm1} Y_4=b_\mp(j) Y_4, \\
T_2Y_1 &= \frac{\xi(b_+(j)-a)}{q(j)-q(i)}(Y_1-\sqrt{-1}Y_2),\quad
T_2Y_2  = \frac{\xi(a-b_-(j))}{q(j)-q(i)}(\sqrt{-1}Y_1-Y_2), \\
T_2Y_3 &= \frac{\xi(b_+(j)-a)}{q(j)-q(i)}(Y_3+Y_4)+\frac{\xi^2(b_+(j)-a)(a-b_-(j)}{(q(j)-q(i))^2}((1-\sqrt{-1})Y_1+(1+\sqrt{-1})Y_2), \\
T_2Y_4 &= \frac{\xi(a-b_-(j))}{q(j)-q(i)}(Y_3-Y_4)+\frac{\xi^2(b_+(j)-a)(a-b_-(j))}{(q(j)-q(i))^2}((-1+\sqrt{-1})Y_1+(1+\sqrt{-1})Y_2), \\
C_3Y_1 &= -C_1Y_2,\quad C_3Y_2=C_1Y_1, \\
C_3Y_3 &= \sqrt{-1}(C_1Y_4)-\xi(1+\sqrt{-1})(C_1Y_2), \\
C_3Y_4 &= \sqrt{-1}(C_1Y_3)+\xi(1-\sqrt{-1})(C_1Y_1).
\end{align*}}
\end{enumerate}
\label{cor_iij}
\end{Cor}

\begin{proof}
It is enough to show the last 4 relations. 
Direct calculations using (\ref{non_triv_eq1}) gives us
\allowdisplaybreaks{
\begin{align*}
C_1Y_1 
&= 
(q(j)-q(i))T_2\otimes\alpha_3 + \xi((a-b_-(j))1\otimes\alpha_3+\sqrt{-1}(a-b_+(j))1\otimes\alpha_4), \\
C_1Y_2 
&= 
(q(j)-q(i))T_2\otimes\alpha_4 + \xi(\sqrt{-1}(a-b_-(j))1\otimes\alpha_3+(a-b_+(j))1\otimes\alpha_4), \\
C_1Y_3
&=
-\sqrt{-1}(q(j)-q(i))T_1T_2\otimes\alpha_3+ (q(j)-q(i))\xi(1+\sqrt{-1})T_2\otimes\alpha_3 + \Delta_{1}, \\
C_1Y_4
&=
\sqrt{-1}(q(j)-q(i))T_1T_2\otimes\alpha_4 + (q(j)-q(i))\xi(1-\sqrt{-1})T_2\otimes\alpha_4 +\Delta_{2}, \\
C_3Y_1 
&= 
(q(j)-q(i))T_2\otimes(-\alpha_4)+\Delta_{3} = -C_1Y_2, \\
C_3Y_2 
&= 
(q(j)-q(i))T_2\otimes\alpha_3+\Delta_{4} = C_1Y_1, \\
C_3Y_3
&= 
-(q(j)-q(i))T_1T_2\otimes\alpha_4+\Delta_{5} 
=
\sqrt{-1}(C_1Y_4)-\sqrt{-1}\xi(1-\sqrt{-1})(C_1Y_2),\\
C_3Y_4
&= 
(q(j)-q(i))T_1T_2\otimes\alpha_3+\Delta_{6} =
\sqrt{-1}(C_1Y_3)-\sqrt{-1}\xi(1+\sqrt{-1})(C_1Y_1).
\end{align*}}
Here $\Delta_1,\cdots,\Delta_6$ are
suitable elements in $\SPAN\{1\otimes\alpha_k\mid1\leq k\leq 4\}(\subseteq M)$.
\end{proof}

\subsection{On the block $[(i,i,i,j)]$ with $|i-j|=1$ and $(\TYPE L(i),\TYPE L(j))=(\TQ,\TM)$}

\begin{Lem}
For any $i,j\in\Z$ such that
\begin{align*}
|i-j|=1,\quad q(j)\ne q(i),\quad (\TYPE L(i),\TYPE L(j))=(\TQ,\TM),
\end{align*}
we define $\MH_4$-supermodule $M$ and $\MH_{3,1}$-supermodule $N$ as follows.
\begin{align*}
M\DEF\IND_{\MH_{3,1}}^{\MH_4}L(iij)\otimes L(i),\quad
N\DEF(X_4+X_4^{-1}-q(i))M\subseteq\RES^{\MH_4}_{\MH_{3,1}}M.
\end{align*}
If $q(j)+2q(i)\ne 0$, then $T_3N\not\subseteq N$ and $M$ is irreducible.
\label{tochuu2}
\end{Lem}

\begin{proof}
By the same reasoning as the Lemma \ref{tochuu},
if $T_3N\not\subseteq N$ then $M$ is irreducible. In the rest of the proof, we show that 
if $q(j)+2q(i)\ne 0$ then $T_3N\not\subseteq N$.

We write $a$ instead of $b_+(i)=b_-(i)$ as in the proof of Lemma \ref{step3} and we adapt 
a basis $\{Y_1,Y_2,Y_3,Y_4\}$ of $L(iij)_{\BARR{0}}$ in Corollary \ref{cor_iij}.
Thus, we can choose
\begin{align*}
\{Z_{\beta,k}\DEF\beta\otimes Y_k\otimes v^i_{\BARR{0}}, W_{\beta,k}\DEF\beta\otimes C_1Y_k\otimes v^i_{\BARR{1}}
\mid \beta\in\{1,T_3,T_2T_3,T_1T_2T_3\},k\in\{1,2,3,4\}\}
\end{align*}
as a basis of $M_{\BARR{0}}$ and 
\begin{align*}
\left\{\substack{Z'_{\beta,k}\DEF(X_4+X_4^{-1}-q(i))Z_{\beta,k}, \\ W'_{\beta,k}\DEF(X_4+X_4^{-1}-q(i))W_{\beta,k}}\mid \beta\in\{T_3,T_2T_3,T_1T_2T_3\},k\in\{1,2,3,4\}\right\}
\end{align*}
as a basis of $N_{\BARR{0}}$. More explicitly, we have
\allowdisplaybreaks{
\begin{align*}
Z'_{T_3,1} &= (q(j)-q(i))T_3\otimes Y_1\otimes v^i_{\BARR{0}}
+\xi((a-b_-(j))1\otimes Y_1\otimes v^i_{\BARR{0}}+(a-b_+(j))1\otimes C_1Y_2\otimes v^i_{\BARR{1}}), \\
Z'_{T_3,2} &=
(q(j)-q(i))T_3\otimes Y_2\otimes v^i_{\BARR{0}}
+\xi((a-b_+(j))1\otimes Y_2\otimes v^i_{\BARR{0}}+(b_-(j)-a)1\otimes C_1Y_1\otimes v^i_{\BARR{1}}), \\
Z'_{T_3,3} &= (q(j)-q(i))T_3\otimes Y_3\otimes v^i_{\BARR{0}}
+\xi((a-b_-(j))1\otimes Y_3\otimes v^i_{\BARR{0}} \\
&{\quad\quad} +\sqrt{-1}(b_+(j)-a)1\otimes C_1Y_4\otimes v^i_{\BARR{1}}
-\xi(1+\sqrt{-1})(b_+(j)-a)1\otimes C_1Y_2\otimes v^i_{\BARR{1}}), \\
Z'_{T_3,4} &= (q(j)-q(i))T_3\otimes Y_4\otimes v^i_{\BARR{0}}
+\xi((a-b_+(j))1\otimes Y_4\otimes v^i_{\BARR{0}} \\
&{\quad\quad}
+\sqrt{-1}(b_-(j)-a)1\otimes C_1Y_3\otimes v^i_{\BARR{1}}
+\xi(1-\sqrt{-1})(b_-(j)-a)1\otimes C_1Y_1\otimes v^i_{\BARR{1}}), \\
W'_{T_3,1} &= (q(j)-q(i))T_3\otimes C_1Y_1\otimes v^i_{\BARR{1}}
+\xi((a-b_-(j))1\otimes C_1Y_1\otimes v^i_{\BARR{1}}+(a-b_+(j))1\otimes Y_2\otimes v^i_{\BARR{0}}), \\
W'_{T_3,2} &= (q(j)-q(i))T_3\otimes C_1Y_2\otimes v^i_{\BARR{1}}
+\xi((a-b_+(j))1\otimes C_1Y_2\otimes v^i_{\BARR{1}}+(b_-(j)-a)1\otimes Y_1\otimes v^i_{\BARR{0}}), \\
W'_{T_3,3} &= (q(j)-q(i))T_3\otimes C_1Y_3\otimes v^i_{\BARR{1}}
+\xi((a-b_-(j))1\otimes C_1Y_3\otimes v^i_{\BARR{1}} \\
&{\quad\quad} +\sqrt{-1}(b_+(j)-a)1\otimes Y_4\otimes v_{i,0}
-\xi(1+\sqrt{-1})(b_+(j)-a)1\otimes Y_2\otimes v_{i,0}), \\
W'_{T_3,4} &= (q(j)-q(i))T_3\otimes C_1Y_4\otimes v^i_{\BARR{1}}
+\xi((a-b_+(j))1\otimes C_1Y_4\otimes v^i_{\BARR{1}} \\
&{\quad\quad}
+\sqrt{-1}(b_-(j)-a)1\otimes Y_3\otimes v^i_{\BARR{0}}
+\xi(1-\sqrt{-1})(b_-(j)-a)1\otimes Y_1\otimes v^i_{\BARR{0}}), \\
Z'_{T_2T_3,k} &= T_2Z'_{T_3,k} = (q(j)-q(i))T_2T_3\otimes Y_k\otimes v^i_{\BARR{0}}+\Delta_k \quad (1\leq k\leq 4), \\
W'_{T_2T_3,k} &= T_2W'_{T_3,k} = (q(j)-q(i))T_2T_3\otimes C_1Y_k\otimes v^i_{\BARR{1}}+\Delta_{k+4} \quad (1\leq k\leq 4).
\end{align*}
}
Here each $\Delta_m$ for $1\leq m\leq 8$ is a suitable element in $\SPAN\{1\otimes C_1^dY_k\otimes v^i_{e}\mid k\in\{1,2,3,4\},d\in\{0,1\},e\in\FF\}(\subseteq M)$ and
write $\Delta_3 = \sum_{k=1}^{4}P_k 1\otimes Y_k\otimes v^i_{\BARR{0}} + \sum_{k=1}^{4}Q_k 1\otimes C_1Y_k\otimes v^i_{\BARR{1}}$ with 
suitable 
coefficients. We define $\Omega_m$, $\Omega_{Z,k}$ and $\Omega_{W,k}$ to be the coefficient of 
$1\otimes Y_1\otimes v^i_{\BARR{0}}$ in $\Delta_m$, $Z'_{T_3,k}$ and $W'_{T_3,k}$ respectively.
Now $T_3Z'_{T_2T_3,3}$ is expanded as follows.
\begin{align*}
{} &{\quad}
\xi(b_+(j)-a)(T_2T_3\otimes Y_3\otimes v^i_{\BARR{0}}+T_2T_3\otimes Y_4\otimes v^i_{\BARR{0}}) \\
&+
\frac{\xi^2(b_+(j)-a)(a-b_-(j))}{q(j)-q(i)}((1-\sqrt{-1})T_2T_3\otimes Y_1\otimes v^i_{\BARR{0}}
+(1+\sqrt{-1})T_2T_3\otimes Y_2\otimes v^i_{\BARR{0}}) \\
&+
\sum_{k=1}^{4}P_k T_3\otimes Y_k\otimes v^i_{\BARR{0}} + \sum_{k=1}^{4}Q_k T_3\otimes C_1Y_k\otimes v^i_{\BARR{1}}.
\end{align*}
Thus, if $T_3Z'_{T_2T_3,3}\in N_{\BARR{0}}$, then we must have
\begin{align*}
T_3Z'_{T_2T_3,3} &=
\frac{\xi(b_+(j)-a)}{q(j)-q(i)}(Z'_{T_2T_3,3}+Z'_{T_2T_3,4}) 
+ \sum_{k=1}^{4}\frac{P_kZ'_{T_3,k} + Q_kW'_{T_3,k}}{q(j)-q(i)}
\\
&{\quad}
+\frac{\xi^2(b_+(j)-a)(a-b_-(j))}{(q(j)-q(i))^2}((1-\sqrt{-1})Z'_{T_2T_3,1}+(1+\sqrt{-1})Z'_{T_2T_3,2}).
\end{align*}
Especially, the coefficient of $1\otimes\alpha_1\otimes v_{i,0}$ of the right hand side must be 0, in other words 
\begin{align*}
{S} &{\DEF}
\frac{\xi(b_+(j)-a)}{q(j)-q(i)}(\Omega_3+\Omega_4)
+ \sum_{k=1}^{4} \frac{P_k\Omega_{Z,k}+Q_k\Omega_{W,k}}{q(j)-q(i)} \\
&+
\frac{\xi^2(b_+(j)-a)(a-b_-(j))}{(q(j)-q(i))^2}((1-\sqrt{-1})\Omega_1+(1+\sqrt{-1})\Omega_2)=0.
\end{align*}
Note that $\Omega_{Z,2}=\Omega_{Z,3}=\Omega_{Z,4}=\Omega_{W,1}=\Omega_{W,3}=0$ and
necessary data are calculated as follows.
\begin{align*}
\Omega_1 &= \frac{\xi^2}{q(j)-q(i)}(a-b_-(j))(b_+(j)-a),\quad
\Omega_2 = \frac{\xi^2\sqrt{-1}}{q(j)-q(i)}(a-b_+(j))(a-b_-(j)), \\
\Omega_3 &= \frac{\xi^3(1-\sqrt{-1})}{(q(j)-q(i))^2}(a-b_-(j))^2(b_+(j)-a),\quad
\Omega_4 = \frac{\xi^3(1-\sqrt{-1})}{(q(j)-q(i))^2}(b_+(j)-a)^2(a-b_-(j)), \\
\Omega_{Z,1} &= \xi(a-b_-(j)),\quad
\Omega_{W,2} = \xi(b_-(j)-a),\quad
\Omega_{W,4} = \xi^2(1-\sqrt{-1})(b_-(j)-a), \\
P_1 &= \Omega_3 = \frac{\xi^3(1-\sqrt{-1})}{(q(j)-q(i))^2}(a-b_-(j))^2(b_+(j)-a), \quad
Q_4 = \frac{-\sqrt{-1}\xi^2}{q(j)-q(i)}(b_+(j)-a)(a-b_-(j)), \\
Q_2 &= \frac{\xi^3(1+\sqrt{-1})\sqrt{-1}}{(q(j)-q(i))^2}(b_+(j)-a)^2(a-b_-(j))+\frac{\xi^3(1+\sqrt{-1})}{q(j)-q(i)}(b_+(j)-a)(a-b_-(j)).
\end{align*}
Using them, we have
\begin{align*}
S = \frac{\xi^4(1-\sqrt{-1})}{(q(j)-q(i))^3}(a-b_-(j))(b_+(j)-a)(4(a-b_-(j))(b_+(j)-a)+(b_+(j)-a)^2+(a-b_-(j))^2).
\end{align*}
Note that $(a-b_-(j))(b_+(j)-a)=aq(j)-2\ne 0$ since $q(j)\ne\pm 2$. Thus, we have 
\begin{align*}
4(a-b_-(j))(b_+(j)-a)+(b_+(j)-a)^2+(a-b_-(j))^2 
= (q(j)+4a)(q(j)-2a) = 0.
\end{align*}
Again, by $q(j)\ne\pm 2$, we have $q(j)+4a=q(j)+2q(i)=0$ if $T_3Z'_{T_2T_3,3}\in N_{\BARR{0}}$.
\end{proof}

\begin{Cor}
Assume $q$ be a primitive $4l$-th root of unity for $l\geq 3$ and $i,j\in\Z$ satisfy 
\begin{align*}
|i-j|=1,\quad q(j)\ne q(i),\quad (\TYPE L(i),\TYPE L(j))=(\TQ,\TM).
\end{align*}
Then we have the following descriptions.
\begin{enumerate}
\item $L(iiji)\cong L(iiij)\cong \IND_{\MH_{3,1}}^{\MH_4}L(iij)\otimes L(i)$.
\item $\CH L(iiji)=\CH L(iiij)=6[L(i)^{\MARU 3}\MARU L(j)]+2[L(i)^{\MARU 2}\MARU L(j)\MARU L(i)]$.
\item $\CH L(jiii)=6[L(j)\MARU L(i)^{\MARU 3}]+2[L(i)\MARU L(j)\MARU L(i)^{\MARU 2}]$.
\item $\CH L(ijii)=2[L(i)\MARU L(j)\MARU L(i)^{\MARU 2}]+2[L(i)^{\MARU 2}\MARU L(j)\MARU L(i)]$.
\end{enumerate}
\label{tochuu3}
\end{Cor}

\begin{proof}
We only need to consider the case $(i,j)=(0,1),(l-1,l-2)$. In each case, we see that $q(j)+2q(i)=0$ implies $q^6=1$.
Thus, we have $L(iiji)\cong\IND_{\MH_{3,1}}^{\MH_4}L(iij)\otimes L(i)$ by Lemma \ref{tochuu2}.
By the same reasoning as Corollary \ref{tochuu_cor}, we have $L(iiij)\cong L(iiji)$.
Note that $L(jiii)\not\cong L(ijii)$ since $L(ji)\not\cong L(ij)$ by Corollary \ref{step2}.
Since $\varepsilon_i(L(iiij)^\sigma)=3$, we see that $L(jiii)\cong L(iiij)^\sigma$.
Now it is easily seen that $L(iiji)\cong \IND_{\MH_{3,1}}^{\MH_4} L(iij)\MARU L(i)$.
\end{proof}

\subsection{The case when $q$ is a primitive $8$-th root of unity}

\begin{Lem}
Let $q$ be a primitive $8$-th root of unity.
We can take a basis $B=\{w_1,w_2\}$ of $L(01)$ such that
$w_1$ is even and $w_2$ is odd and the matrix representations with respect to $B$ is as follows.
\begin{align*}
X_1^{\pm1} : \begin{pmatrix}
1 & 0 \\
0 & 1 \\
\end{pmatrix},
X_2^{\pm1} : \begin{pmatrix}
-1 & 0 \\
0 & -1 \\
\end{pmatrix},
C_1 : \begin{pmatrix}
0 & 1 \\
1 & 0 \\
\end{pmatrix},
C_2 : \begin{pmatrix}
0 & -q^2 \\
q^2 & 0 \\
\end{pmatrix},
T_1 : \begin{pmatrix}
q & 0 \\
0 & q^3 \\
\end{pmatrix}.
\end{align*}
\label{01}
\end{Lem}

\begin{proof}
We can check by direct calculation that they satisfy the defining relations of $\MH_2$.
It is clearly irreducible and note that the whole space is a simultaneous $(2,-2)=(q(0),q(1))$-eigenspace of 
$(X_1+X_1^{-1},X_2+X_2^{-1})$. 
\end{proof}

\begin{Cor}
We have $\CH L(01)=[L(0)\MARU L(1)]$ and $\CH L(10)=[L(1)\MARU L(0)]$.
\label{8th_1}
\end{Cor}

\allowdisplaybreaks{
%

\begin{Lem}
Let $q$ be a primitive $8$-th root of unity.
We can take a  basis $B=\{w_i\mid 1\leq i\leq 8\}$ of $L(001)$ such that
$w_{i}$ is even and $w_{i+4}$ is odd for $1\leq i\leq 4$ and the matrix representations with respect to $B$ is as follows.
\allowdisplaybreaks{
\begin{align*}
X_i &: \begin{pmatrix}
M_{X_i} & O \\
O & M_{X_i} \\
\end{pmatrix}, \quad
X_3^{\pm1} : -E_8, \quad
X_1^{-1} : 2E_8-X_1, \quad
X_2^{-1} : 2E_8-X_2, \\
C_j &: \begin{pmatrix}
O & M_{C_j} \\
-M_{C_j} & O \\
\end{pmatrix}, \quad
T_1 : \frac{1}{1+q^2}\begin{pmatrix}
M_{T_1} & O \\
O & M_{T_1} \\
\end{pmatrix}, \quad
T_2 : \begin{pmatrix}
M_{T_2} & O & O & O \\
O & M_{T_2} & O & O \\
O & O & M_{T_2} & O \\
O & O & O & M_{T_2} \\
\end{pmatrix},
\end{align*}
}
for $1\leq i\leq 2$ and $1\leq j\leq 3$ where 
\allowdisplaybreaks{
\begin{align*}
M_{X_1} &= \begin{pmatrix}
1 & 0 & -2 & 2q \\
0 & 1 & 2q & -2q^2 \\
2 & 2q^{-1} & 1 & 0 \\
2q^{-1} & -2q^2 & 0 & 1 \\
\end{pmatrix}, \quad
M_{X_2} = \begin{pmatrix}
-1 & -2q^{-1} & 0 & 0 \\
2q & 3 & 0 & 0 \\
0 & 0 & -1 & 2q \\
0 & 0 & -2q^{-1} & 3 \\
\end{pmatrix}, \\
M_{C_1} &= \begin{pmatrix}
q^2 & 0 & 2q^2 & 2q^{-1} \\
0 & q^2 & 2q^{-1} & -2 \\
2q^2 & 2q & -q^2 & 0 \\
2q & 2 & 0 & -q^2 \\
\end{pmatrix}, \quad
M_{C_2} = \begin{pmatrix}
0 & 0 & q^2 & 0 \\
0 & 0 & 2q^{-1} & -1 \\
q^2 & 0 & 0 & 0 \\
2q & 1 & 0 & 0 \\
\end{pmatrix}, \\
M_{C_3} &= \begin{pmatrix}
0 & 0 & -1 & 0 \\
0 & 0 & 0 & q^2 \\
1 & 0 & 0 & 0 \\
0 & q^2 & 0 & 0 \\
\end{pmatrix}, \quad
M_{T_1} = \begin{pmatrix}
q^3 & q^2 & -q^3 & -1 \\
0 & q^3 & 0 & q \\
q^3 & q^2 & q^3 & 1 \\
0 & q & 0 & q^3 \\
\end{pmatrix}, \quad
M_{T_2} = \begin{pmatrix}
q^3+q & 1 \\
1 & 0 \\
\end{pmatrix}.
\end{align*}
}
\label{001}
\end{Lem}

\begin{proof}
We can check by direct calculation that they satisfy the defining relations of $\MH_3$ 
and whole space is a simultaneous $(2,2,-2)=(q(0),q(0),q(1))$-eigenspace of 
$(X_1+X_1^{-1},X_2+X_2^{-1},X_3+X_3^{-1})$.
Since $\dim L(0^2)\MARU L(1)=8$ and Theorem \ref{kato}, this supermodule is irreducible.
\end{proof}

\begin{Cor}
Let $q$ be a primitive $8$-th root of unity.
We have $\CH L(001)=2[L(0)^{\MARU2}\MARU L(1)]$, $\CH L(010)=[L(0)\MARU L(1)\MARU L(0)]$ and 
$\CH L(100)=2[L(1)\MARU L(0)^{\MARU2}]$.
\label{8th_2}
\end{Cor}

\begin{proof}
By $\CH L(001)=2[L(0)^{\MARU2}\MARU L(1)]$, we have $L(100)\cong L(001)^\sigma$.
Consider $M=\IND_{\MH_{2,1}}^{\MH_3}L(01)\MARU L(0)$. By Corollary \ref{8th_1} and Lemma \ref{char_lemma}, 
we have $\CH M=[L(0)\MARU L(1)\MARU L(0)]+2[L(0)^{\MARU2}\MARU L(1)]$.
Apply Theorem \ref{Kato_irr_thm} (\ref{Kato_irr_thm_1st}), we see that $L(010)\cong \HEAD M$ with
$\CH L(010)=[L(0)\MARU L(1)\MARU L(0)]$.
\end{proof}

\begin{Cor}
Let $q$ be a primitive $8$-th root of unity. 
Then $M\DEF\IND^{\MH_4}_{\MH_{3,1}}L(001)\MARU L(0)$ is an irreducible $\MH_4$-supermodule.
\label{8th_3}
\end{Cor}

\begin{proof}
Take a  basis $\{w_i\mid1\leq i\leq 8\}$ in Lemma \ref{001}.
Consider the following linear transformations with respect to this basis.
\begin{align*}
X_4^{\pm1} : E_8,\quad
C_4 : \begin{pmatrix}
O & -E_4 \\
-E_4 & O \\
\end{pmatrix}.
\end{align*}
We can check 
that the matrix representations of
$\{X_i^{\pm1},C_i,T_j\mid 1\leq i\leq 4, 1\leq j\leq 3\}$ satisfy the defining relations of 
$\MH_{3,1}$. Thus, they are also matrix representations of $L(001)\MARU L(0)$.

To prove that $M$ is irreducible, it is enough to show that $\MH_{3,1}$-supermodule 
$N\DEF (X_4+X_4^{-1}-q(0))M$ is not $T_3$-invariant as in the proof of Lemma \ref{tochuu2}.
Thus, it is enough to show that $T_3Z\ne (Z-W)/2$ where
\begin{align*}
Z &\DEF (X_4+X_4^{-1}-2)T_3\otimes w_1 = -4T_3\otimes w_1 + 2\xi(w_1+w_3), \\
W &\DEF (X_4+X_4^{-1}-2)T_3\otimes w_3 = -4T_3\otimes w_3 + 2\xi(w_3-w_1), \\
T_3Z &= -2\xi(T_3\otimes w_1 - T_3\otimes w_3)-4\cdot 1\otimes w_1.
\end{align*}
This follows from $2\xi\ne-4$.
\end{proof}

\begin{Cor}
Let $q$ be a primitive $8$-th root of unity.
Then we have the following descriptions.
\begin{enumerate}
\item $L(0010)\cong L(0001)\cong \IND_{\MH_{3,1}}^{\MH_4}L(001)\MARU L(0)$.
\item $\CH L(0010)=\CH L(0001)=6[L(0)^{\MARU 3}\MARU L(1)]+2[L(0)^{\MARU 2}\MARU L(1)\MARU L(0)]$.
\item $\CH L(1000)=6[L(1)\MARU L(0)^{\MARU 3}]+2[L(0)\MARU L(1)\MARU L(0)^{\MARU 2}]$.
\item $\CH L(0100)=2[L(0)\MARU L(1)\MARU L(0)^{\MARU 2}]+2[L(0)^{\MARU 2}\MARU L(1)\MARU L(0)]$.
\end{enumerate}
\end{Cor}

\begin{proof}
Same as the proof of Corollary \ref{tochuu3}.
\end{proof}

\section{Hecke-Clifford superalgebras and crystals of type $D^{(2)}_{l}$}
\label{families}
Recall so far that $F$ is an algebraically closed field of characteristic different from $2$.
From now on, we 
assume 
that $q$ is a primitive $4l$-th root of unity for $l\geq 2$
and choose $\{0,1,\cdots,l-1\}$ as $I_q$.
Note that we have $q(0)=2$ and $q(l-1)=-2.$

\subsection{Lie theory of type $D^{(2)}_l$}
Consider the Dynkin diagram and the affine Cartan matrix indexed by $I_q$
of type $D^{(2)}_l$ as follows\footnote{According to Kac's notation~\cite[TABLE Aff 1-3]{Kac},
$D^{(2)}_2$ 
should be regarded as
$A^{(1)}_1$.}.

\begin{center}
\begin{minipage}[m]{225pt}
\includegraphics{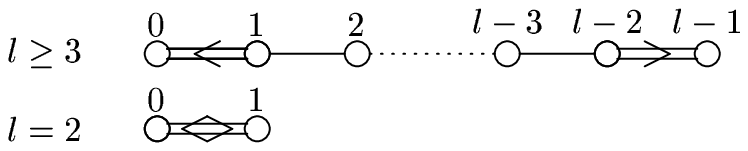}
\end{minipage} 
\begin{minipage}[m]{125pt}
\includegraphics[scale=0.72]{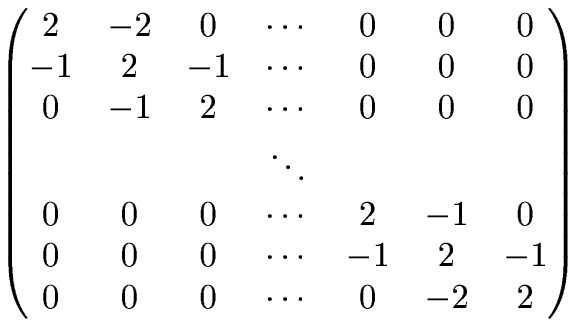}
\end{minipage}
\end{center}

In the rest of this section, let $\GEE$ be the corresponding Kac-Moody Lie algebra and
apply definitions in \S\ref{general_setting} for $A=D^{(2)}_l$.


\subsection{Representations of low rank affine Hecke-Clifford superalgebras}
\label{onaji_lemma}
The purpose of this subsection is to show that ~\cite[Lemma 5.19, Lemma 5.20]{BK}
still hold in our setting, i.e., when $q$ is a primitive $4l$-th root of unity for $l\geq 2$.
This fact is responsible for the appearance of the Lie theory of type $D^{(2)}_l$.

\begin{Lem}
Let $i,j\in I_q$ with $|i-j|=1$. Then, for all $a,b\geq 0$ with $a+b<-\langle h_i,\alpha_j \rangle$,
there is a non-split short exact sequence
\begin{align}
0\longrightarrow L(i^{a+1}ji^b)\longrightarrow
\IND^{\MH_{a+b+2}}_{\MH_{a+b+1,1}}L(i^aji^b)\MARU L(i)\longrightarrow L(i^aji^{b+1})\longrightarrow 0.
\label{ses}
\end{align}
Moreover, for every $a,b\geq 0$ with $a+b\leq -\langle h_i,\alpha_j \rangle$, we have
\begin{align}
\CH L(i^aji^b) = a!b![L(i)^{\MARU a}\MARU L(j)\MARU L(i)^{\MARU b}].
\label{char_final}
\end{align}
\label{ess_1}
\end{Lem}

\begin{proof}
(\ref{char_final}) is established in Corollary \ref{step2}, Corollary \ref{cor_iij}, 
Corollary \ref{8th_1} and Corollary \ref{8th_2}.
An existence of a non-split short exact sequence (\ref{ses}) 
follows from Lemma \ref{char_lemma}, Theorem \ref{Kato_irr_thm} (\ref{Kato_irr_thm_1st}), Definition \ref{def_of_L} 
and the injectivity of the formal character 
map $\CH:\KKK(\REP\MH_n)\hookrightarrow\KKK(\REP\MA_n)$~\cite[Theorem 5.12]{BK}.
\end{proof}

\begin{Lem}
Let $i,j\in I_q$ with $|i-j|=1$ and set $n=1-\langle h_i,\alpha_j \rangle$. 
Then $L(i^nj)\cong L(i^{n-1}ji)$. Moreover, for every $a,b\geq 0$ with $a+b=-\langle h_i,\alpha_j \rangle$,
we have
\begin{align*}
L(i^aji^{b+1})\cong\IND^{\MH_{n+1}}_{\MH_{n,1}}L(i^aji^b)\MARU L(i)\cong
\IND^{\MH_{n+1}}_{\MH_{1,n}}L(i)\MARU L(i^aji^b)
\end{align*}
with character 
\begin{align*}
a!(b+1)![L(i)^{\MARU a}\MARU L(j)\MARU L(i)^{\MARU (b+1)}]+(a+1)!b![L(i)^{\MARU (a+1)}\MARU L(j)\MARU L(i)^{\MARU b}].
\end{align*}
\label{ess_2}
\end{Lem}

\begin{proof}
Character formulas
are established in Corollary \ref{tochuu_cor}, Corollary \ref{tochuu3} and Corollary \ref{8th_3}.
The rest of reasoning is the same as the proof of Lemma \ref{ess_1}.
\end{proof}

\begin{Cor}
The operators $\{e_i:K(\infty)\to K(\infty)\mid i\in I_q\}$ satisfy the Serre relations, i.e.,
\begin{align}
\begin{split}
\begin{cases}
e_ie_j = e_je_i & \textrm{if $|i-j|>1$}, \\
e_i^2e_j+e_je_i^2=2e_ie_je_i & \textrm{if $|i-j|=1$ and $i\ne 0$ and $i\ne l-1$}, \\
e_i^3e_j+3e_ie_je_i^2=3e_i^2e_je_i+e_je_i^3 & \textrm{otherwise}.
\end{cases}
\end{split}
\label{serre_rel2}
\end{align}
\label{serre_rel}
\end{Cor}

\begin{proof}
By Lemma \ref{divided_root_op} and coassociativity of $\Delta$,
it is enough to check the same relation on $\KKK(\REP\MH_2), \KKK(\REP\MH_3)$ and $\KKK(\REP\MH_4)$ respectively.
It is achieved using the character formulas in Lemma \ref{ess_1} and Lemma \ref{ess_2}.
\end{proof}

The same arguments using Lemma \ref{ess_1} and Lemma \ref{ess_2}
establishes the following~\cite[Lemma 5.23]{BK2}.

\begin{Lem}
Take $M\in\IRR(\REP\MH_n)$ and $i,j\in I_q$ with $i\ne j$.
Then the followings hold where $k=-\langle h_i,\alpha_j\rangle$ and $\varepsilon=\varepsilon_i(M)$.
\begin{enumerate}
\item There exists a unique pair of non-negative integers $(a,b)$ with $a+b=k$ such that
for every $m\geq 0$ we have
$\varepsilon_i(\Tf^m_i\Tf_j M) = m + \varepsilon - a$.
\label{very_important2_1}
\item $[\HEAD\IND \Tf^{m-k}_iM\MARU L(i^aji^b):\Tf^m_i\Tf_j M]>0$ for $m\geq k$.
\label{very_important2_2}
\item $[\HEAD\IND \Te^{k-m}_iM\MARU L(i^aji^b):\Tf^m_i\Tf_j M]>0$ for $0\leq m<k\leq m+\varepsilon$.
\label{very_important2_3}
\end{enumerate}
\label{very_important2}
\end{Lem}

Note that Lemma \ref{very_important2} (\ref{very_important2_2}) and (\ref{very_important2_3})
is equivalent to saying that we have
\begin{align*}
[\HEAD\IND (\Tf^{\varepsilon+m-k}_i\Te^\varepsilon_iM)\MARU L(i^aji^b):\Tf^m_i\Tf_j M]>0
\end{align*}
for every $m\geq 0$ with $k\leq m+\varepsilon$.

Keep the setting in Lemma \ref{very_important2}. Since there are surjections 
\begin{align*}
\IND \Te^\varepsilon_iM\MARU L(i^{\varepsilon+m-k})\longtwoheadrightarrow \Tf^{\varepsilon+m-k}_i\Te^\varepsilon_iM,\quad
\IND L(i^a)\MARU L(ji^b)\longtwoheadrightarrow L(i^aji^b)
\end{align*}
by Theorem \ref{Kato_irr_thm} (\ref{Kato_irr_thm_1st}) and Lemma \ref{ess_1} respectively, we have
\begin{align*}
[\HEAD\IND (\Te^\varepsilon_iM\MARU L(i^{\varepsilon+m-b})\MARU L(ji^b)):\Tf^m_i\Tf_j M]>0.
\end{align*}
By Frobenius reciprocity there is a non-zero injective homomorphism
\begin{align*}
\Te^\varepsilon_iM\MARU L(i^{\varepsilon+m-b})\MARU L(ji^b)\longhookrightarrow \RES_{\MH_{n-\varepsilon,\varepsilon+m-b,b+1}}\Tf^m_i\Tf_j M.
\end{align*}
Thus, we also have a non-zero injective homomorphism
\begin{align*}
\Te^\varepsilon_iM\MARU L(i^{\varepsilon+m-b})\longhookrightarrow \RES_{\MH_{n-\varepsilon,\varepsilon+m-b}}\Tf^m_i\Tf_j M.
\end{align*}
Again by Frobenius reciprocity, for every $m\geq 0$ with $k\leq m+\varepsilon$ we have
\begin{align}
[\RES_{\MH_{n+m-b}}\Tf^m_i\Tf_j M:\Tf^{m-b}_iM]>0.
\label{incl}
\end{align}

\subsection{Cyclotomic Hecke-Clifford superalgebra}

\begin{Def}
For each positive integral weight $\lambda\in P^{+}$, we define a polynomial 
\begin{align*}
f^{\lambda}=(X_1-1)^{\lambda(h_0)}(X_1+1)^{\lambda(h_{l-1})}\prod_{i=1}^{l-2}(X_1^2-q(i)X_1+1)^{\lambda(h_i)}
\end{align*}
\end{Def}

Note that since the canonical central element 
is $c=h_0+h_{l-1}+\sum_{i=1}^{l-2}2h_i$, the degree of $f^{\lambda}$ is $\lambda(c)$.
We can easily check that $f^{\lambda}$ satisfies the assumption in Definition \ref{cyc_quo}.
From now on, we apply all the constructions in \S\ref{cyc_quot_section} for $\EF=f^\lambda$
and abbreviate $K(\EF)$, $e^\EF_i$, etc.\ to $K(\lambda)$, $e^\lambda_i$, etc.\ respectively.

As a Corollary of Lemma \ref{kill_pr},
we have the following characterization 
of $\IM(\INFL^\lambda:B(\lambda)\hookrightarrow B(\infty))$~\cite[Corollary 6.13]{BK}.

\begin{Cor}
Let $\lambda\in P^+$ and $M\in B(\infty)$.
We have $\PR^\lambda M=M$ if and only if $\varepsilon^*_i(M)\leq \lambda(h_i)$ for all $i\in I_q$.
\label{kill_irr}
\end{Cor}

\begin{Lem}
Let $i,j\in I_q$ with $i\ne j$ and $M\in \IRR(\SMOD{\MH^\lambda_n})$ such that $\varphi^\lambda_j(M)>0$.
Then $\varphi^\lambda_i(\Tf^\lambda_jM)-\varepsilon^\lambda_i(\Tf^\lambda_jM)\leq \varphi^\lambda_i(M)-\varepsilon^\lambda_i(M)-a_{ij}$.
\label{very_important}
\end{Lem}

\begin{proof}
Put $\varepsilon=\varepsilon^\lambda_i(M)=\varepsilon_i(\INFL^\lambda M)$.
Apply Lemma \ref{very_important2} to $\INFL^\lambda M$ and 
take a pair $(a,b)$ in Lemma \ref{very_important2} (\ref{very_important2_1}).
Since $\varepsilon^\lambda_i(\Tf^\lambda_jM)=\varepsilon_i(\Tf_j\INFL^\lambda M)=\varepsilon-a$, it is enough to show
that $\varphi^\lambda_i(\Tf^\lambda_jM)\leq \varphi^\lambda_i(M)+b$.
Note that $m>\varphi^\lambda_i(M)+b$ implies that $-a_{ij}\leq m+\varepsilon$ 
by $m+\varepsilon+a_{ij}>\varphi^\lambda_i(M)+(\varepsilon-a)$. 
Thus, we have
\begin{align*}
\varepsilon^*_i(\Tf^m_i\Tf_j \INFL^\lambda M)\geq \varepsilon^*_i(\Tf^{m-b}_i\INFL^\lambda M)> \lambda(h_i).
\end{align*}
Here the 1st inequality follows from (\ref{incl}) and the 2nd inequality follows from
Corollary \ref{kill_irr} and $\sigma$-version of Lemma \ref{comm_cry} (\ref{comm_cry2}).
Again by Corollary \ref{kill_irr}, we have $\PR^\lambda \Tf^m_i\Tf_j \INFL^\lambda M=0$ 
for each $m>\varphi^\lambda_i(M)+b$, i.e., $\varphi^\lambda_i(\Tf^\lambda_jM)\leq \varphi^\lambda_i(M)+b$.
\end{proof}

\begin{Thm}
For any $M\in\IRR(\SMOD{\MH^\lambda_n})$ and $i\in I_q$, 
we have $\varphi^\lambda_i(M)-\varepsilon^\lambda_i(M)=\langle h_i, \lambda+\WT(\INFL^\lambda M)\rangle$.
\label{weight}
\end{Thm}

\begin{proof}
By Corollary \ref{kill_irr}, 
we have $\varphi^\lambda_i(\TRIVREP_\lambda)=\lambda(h_i)$.
Combined with the obvious $\varepsilon^\lambda_i(\TRIVREP_\lambda)=0$ and Lemma \ref{very_important} inductively, 
we have $\varphi^\lambda_i(M)-\varepsilon^\lambda_i(M)\leq\langle h_i, \lambda+\WT(\INFL^\lambda M)\rangle$.
Thus, it is enough to show that 
\begin{align*}
(\varphi^\lambda_0(M)-\varepsilon^\lambda_0(M))+(\varphi^\lambda_{l-1}(M)-\varepsilon^\lambda_{l-1}(M))+
\sum_{i=1}^{l-2}2(\varphi^\lambda_i(M)-\varepsilon^\lambda_i(M))=\lambda(h_i),
\end{align*}
which is the same thing as Corollary \ref{nilp3}.
\end{proof}

\begin{Cor}
The 6-tuple
$(B(\lambda),\WT^\lambda,\{\varepsilon^\lambda_i\}_{i\in I_q}, \{\varphi^\lambda_i\}_{i\in I_q},\{\Te^\lambda_i\}_{i\in I_q},\{\Tf^\lambda_i\}_{i\in I_q})$
is a $\GEE$-crystal by defining $\WT^\lambda(M)=\lambda+\WT(\INFL^\lambda M)$ for $M\in B(\lambda)$.
\end{Cor}

\subsection{Lie-theoretic descriptions of $B(\infty)$ and $B(\lambda)$}
\label{identify2}

\begin{Thm}
For each $i\in I_q$, the map 
\begin{align*}
\Psi_i:B(\infty)\longrightarrow B(\infty)\otimes B_i,\quad
[M]\longmapsto [(\Te^*_i)^{\varepsilon^*_i(M)}M]\otimes b_i(-\varepsilon^*_i(M))
\end{align*}
is a crystal embedding.
\label{crys_emd}
\end{Thm}

\begin{proof}
We prove $\Psi_i([\Tf_j M])=\Tf_j\Psi_i([M])$ 
for any $i,j\in I_q$ and $[M]\in B(\infty)$.
In case $i\ne j$, this follows from $\sigma$-versions of Lemma \ref{comm_cry} (\ref{comm_cry2}) and (\ref{comm_cry3}).

Let us assume $i=j$ and put $a=\varepsilon^*_i(M)$. By Definition \ref{cry_tensor}, 
\begin{align*}
\Tf_i\Psi_i([M]) = 
\begin{cases}
[\Tf_i(\Te^*_i)^{a}M]\otimes b_i(-a) & \textrm{if $\varepsilon_i((\Te^*_i)^{a}M)+a+\langle h_i,\WT(M) \rangle > 0$}, \\
[(\Te^*_i)^{a}M]\otimes b_i(-a-1) & \textrm{if $\varepsilon_i((\Te^*_i)^{a}M)+a+\langle h_i,\WT(M) \rangle \leq 0$}.
\end{cases}
\end{align*}
Comparing with $\sigma$-versions of Lemma \ref{comm_cry} (\ref{comm_cry1}), (\ref{comm_cry3}) and (\ref{comm_cry4}), 
it is enough to show the following.
\begin{align*}
\varepsilon^*_i(\Tf_i M) = 
\begin{cases}
a & \textrm{if $\varepsilon_i((\Te^*_i)^{a}M)+a+\langle h_i,\WT(M) \rangle > 0$}, \\
a+1 & \textrm{if $\varepsilon_i((\Te^*_i)^{a}M)+a+\langle h_i,\WT(M) \rangle \leq 0$}.
\end{cases}
\end{align*}

Consider the case $\varepsilon_i((\Te^*_i)^{a}M)+a+\langle h_i,\WT(M) \rangle > 0$ and 
take $\lambda_1\in P^+$ such that $\lambda_1(h_j)$ is big enough for any $j\ne i$ and $\lambda_1(h_i)=a$.
Note that $M$ can be regarded as an element of $B(\lambda_1)$ by Corollary \ref{kill_irr}.
By Theorem \ref{weight}, we have
\begin{align*}
\varphi^{\lambda_1}_i(\PR^{\lambda_1} M)
=
\varepsilon^{\lambda_1}_i(\PR^{\lambda_1} M)+\langle h_i,\lambda_1+\WT(M) \rangle
&=
\varepsilon_i(M)+a+\langle h_i,\WT(M) \rangle \\
&\geq
\varepsilon_i((\Te^*_i)^{a}M)+a+\langle h_i,\WT(M) \rangle 
\geq 1.
\end{align*}
Thus, we have $\varepsilon^*_i(\Tf_i M)\leq \lambda_1(h_i)=a$ by Corollary \ref{kill_irr}.
It implies $\varepsilon^*_i(\Tf_i M)=a$ by $\sigma$-version of Lemma \ref{comm_cry} (\ref{comm_cry1}).

Finally, consider the case $\varepsilon_i((\Te^*_i)^{a}M)+a+\langle h_i,\WT(M) \rangle \leq 0$, i.e.,
\begin{align*}
\varepsilon^*_i((\Te_i)^{a}M^\sigma)+a+\langle h_i,\WT(M^\sigma) \rangle 
=\varepsilon^*_i((\Te_i)^{a}M^\sigma)-a+\langle h_i,\WT((\Te_i)^{a}M^\sigma) \rangle 
\leq 0.
\end{align*}
Take $\lambda_2\in P^+$ such that $\lambda_2(h_j)$ is big enough for any $j\ne i$
and $\lambda_2(h_i)=r+\varepsilon^*_i((\Te_i)^{a}M^\sigma)$ for
$r=a-\varepsilon^*_i((\Te_i)^{a}M^\sigma)-\langle h_i,\WT((\Te_i)^{a}M^\sigma) \rangle (\geq 0)$.
Again $(\Te_i)^{a}M^\sigma$ can be regarded as an element of $B(\lambda_2)$ and we have
\begin{align*}
\varphi^{\lambda_2}_i(\PR^{\lambda_2} (\Te_i)^{a}M^\sigma)
&=
\varepsilon^{\lambda_2}_i(\PR^{\lambda_2} (\Te_i)^{a}M^\sigma)+\langle h_i,\lambda_2+\WT((\Te_i)^{a}M^\sigma) \rangle \\
&=
\langle h_i,\lambda_2+\WT((\Te_i)^{a}M^\sigma) \rangle
=
a
\end{align*}
by Theorem \ref{weight}.
Combined with Corollary \ref{kill_irr}, it implies
\begin{align*} 
\begin{cases}
\varepsilon_i(M)=\varepsilon^*_i(M^\sigma)
=\varepsilon^*_i(\Tf_i^{a}(\Te_i)^{a}M^\sigma)\leq \lambda_2(h_i), \\
\varepsilon_i(\Tf^*_i M)=\varepsilon^*_i(\Tf_i M^\sigma)
=\varepsilon^*_i(\Tf_i^{a+1}(\Te_i)^{a}M^\sigma)\geq \lambda_2(h_i)+1. \\
\end{cases}
\end{align*}
Thus, by Lemma \ref{comm_cry} (\ref{comm_cry1}), we have 
\begin{align*}
\varepsilon_i(M)=\lambda_2(h_i)=
a-\langle h_i,\WT((\Te_i)^{a}M^\sigma) \rangle
=-a-\langle h_i,\WT(M) \rangle.
\end{align*}
Take $\lambda_3\in P^+$ such that $\lambda_3(h_j)$ is big enough for any $j\ne i$
and $\lambda_3(h_i)=a$.
Again $M$ can be regarded as an element of $B(\lambda_3)$ and we have
\begin{align*}
\varphi^{\lambda_3}_i(\PR^{\lambda_3} M)
=
\varepsilon^{\lambda_3}_i(\PR^{\lambda_3} M)+\langle h_i,\lambda_3+\WT(M) \rangle
=
\varepsilon_i(M)+a+\langle h_i,\WT(M) \rangle
=0
\end{align*}
by Theorem \ref{weight}.
Thus, we have $\varepsilon^*_i(\Tf_i M)> \lambda_3(h_i)=a$ by Corollary \ref{kill_irr}.
It implies $\varepsilon^*_i(\Tf_i M)=a+1$ by $\sigma$-version of Lemma \ref{comm_cry} (\ref{comm_cry1}).
\end{proof}

\begin{Cor}
The $\GEE$-crystal $B(\infty)$ is isomorphic to $\B(\infty)$.
\label{final_thm2}
\end{Cor}

\begin{proof}
Apply Proposition \ref{characterization_thm1} to $B=B(\infty)$ and $b_0=[\TRIVREP]$.
\end{proof}

\begin{Cor}
For each $\lambda\in P^+$, 
the $\GEE$-crystal $B(\lambda)$
is isomorphic to $\B(\lambda)$.
\label{final_thm1}
\end{Cor}

\begin{proof}
Apply Proposition \ref{characterization_thm2} to $B=B(\lambda)$, $b_\lambda=[\TRIVREP_\lambda]$ and a map 
\begin{align*}
\Phi:
B(\infty)\otimes T_\lambda\longrightarrow B(\lambda),\quad
[M]\otimes t_{\lambda}\longmapsto [\PR^\lambda M].
\end{align*}
The latter is an $f$-strict crystal morphism since $\Tf^\lambda_i=\PR^\lambda\circ\Tf_i\circ\INFL^\lambda$ by Definition \ref{def_of_kashiwara_cyc} and $\Tf_i M\ne 0$ for any $M\in B(\infty)$ by Definition \ref{kashiwara_def}.
\end{proof}

\subsection{Lie-theoretic descriptions of $K(\infty)_\Q$ and $K(\lambda)_\Q$}
\label{identify1}
\begin{Thm}
For each $\lambda\in P^+$, we have the followings.
\begin{enumerate}
\item $K(\lambda)_\Q$ has a left $U_\Q(=\langle e_i,f_i,h_i\mid(\ref{envelop})\rangle_{i\in I_q})$-module structure by 
\begin{align*}
e_i[M] = [e^\lambda_iM],\quad
f_i[M] = [f^\lambda_iM],\quad
h_i[M] = \langle h_i,\WT^\lambda(M)\rangle[M],
\end{align*}
and it is isomorphic to the integrable highest weight $U_\Q$-module of highest weight $\lambda$ with 
highest weight vector $[\TRIVREP_\lambda]$.
\label{final_thm4_1}
\item The symmetric non-degenerate bilinear form $\langle,\rangle_{\lambda}$ on $K(\lambda)_\Q$ in \S\ref{symmetric_form} 
coincides with the usual Shapovalov form satisfying $\langle [\TRIVREP_\lambda],[\TRIVREP_\lambda]\rangle_{\lambda}=1$ under
the above identification.
\item $\bigoplus_{n\geq 0}\KKK(\PROJ\MH^\lambda_n)\cong K(\lambda)^*\subseteq K(\lambda)$ 
are two integral lattices of $K(\lambda)_\Q$ containing $[\TRIVREP_\lambda]$ with $K(\lambda)^*=U^-_\Z[\TRIVREP_\lambda]$
and $K(\lambda)$ being its dual under the Shapovalov form.
\end{enumerate}
\label{final_thm3}
\end{Thm}

\begin{proof}
By \S\ref{mod_str_cyc} and Corollary \ref{serre_rel}, 
the operators $\{e^\lambda_i:K(\lambda)\to K(\lambda)\mid i\in I_q\}$ satisfy the Serre relations (\ref{serre_rel2}).
It implies that the operators $\{f^\lambda_i:K(\lambda)^*\to K(\lambda)^*\mid i\in I_q\}$ 
satisfy the Serre relations by Lemma \ref{adj_form}.
Thus, both operators satisfy the Serre relations on $K(\lambda)_\Q$ by Theorem \ref{two_lattice}.
By Corollary \ref{nilp2} and Theorem \ref{weight}, 
we have $[e^\lambda_i,f^\lambda_j]=\delta_{i,j}h_i$ as operators on $K(\lambda)_\Q$.
Since other relations of (\ref{envelop}) are immediately deduced from the definition of the action of $h_i$, 
$K(\lambda)_\Q$ has a left $U_\Q$-module structure by the above actions.
By Corollary \ref{nilp1}, $e^\lambda_i$ and $f^\lambda_i$ are both nilpotent operators on $K(\lambda)_\Q$.
Since the action of $\{h_i\mid i\in I_q\}$ is diagonalized with finite-dimensional weight spaces
by the definition, 
$K(\lambda)_\Q$ is an integrable $U_\Q$-module.
By Theorem \ref{proj_lattice}, $\KKK(\lambda)_\Q=U^-_{\Q}[\TRIVREP_\lambda]$ is a 
highest weight $U_\Q$-module of highest weight $\lambda$ with highest weight vector $[\TRIVREP_\lambda]$.
Now (ii) is a direct consequence of Lemma \ref{adj_form} and Corollary \ref{sym_bil} and (iii) is a restatement of 
Theorem \ref{two_lattice} and Corollary \ref{proj_lattice}.
\end{proof}

\begin{Thm}
There exists a graded $\Z$-Hopf algebra isomorphism $U^+_\Z\ISOM K(\infty)^*$ which
takes $e^{(r)}_i$ to $\delta_{L(i^r)}$ for each $i\in I_q$ and $r\geq 0$.
\label{final_thm4}
\end{Thm}

\begin{proof}
By \S\ref{module_str_inf} and Corollary \ref{serre_rel}, 
there exists a graded $\Z$-algebra map $\pi:U^+_\Z\to K(\infty)^*$
which takes $e^{(r)}_i$ to $\delta_{L(i^r)}$ for each $i\in I_q$ and $r\geq 0$.
It is easily checked that it is a graded $\Z$-coalgebra map since $\delta_{L(i)}$ is mapped to
$\delta_{L(i)}\otimes 1+1\otimes \delta_{L(i)}$ via the comultiplication of $K(\infty)^*$.
Thus, $\pi$ is a graded $\Z$-Hopf algebra map by ~\cite[Lemma 4.0.4]{Swe}.

It is enough to show that $\pi$ is an isomorphism as graded $\Z$-modules.
By Corollary \ref{kill_irr}, we have a natural 
isomorphism $\varinjlim_{\lambda\in P^+}\KKK(\SMOD{\MH^\lambda_n})\ISOM\KKK(\REP\MH_n)$. 
Combined with Theorem \ref{proj_lattice}, it gives us
\begin{align*}
\HOM_{\Z}(\KKK(\REP\MH_n),\Z)
&\cong
\varprojlim_{\lambda\in P^+}\HOM_{\Z}(\KKK(\SMOD{\MH^\lambda_n}),\Z) \\
&\cong 
\varprojlim_{\lambda\in P^+}\KKK(\PROJ\MH^\lambda_n)
=
\varprojlim_{\lambda\in P^+}(U^{-}_{\Z})_{n}[\TRIVREP_{\lambda}]
\COISOM
(U^{-}_{\Z})_{n},
\end{align*}
where $(U^{-}_{\Z})_{n}$ is the set of homogeneous elements of $U^-_{\Z}$ of degree $n$ via the principal
grading, i.e., $\DEG f^{(r)}_i=r$ for all $i\in I_q$ and $r\geq 0$.
The last isomorphism follows easily from the fact 
$(U^{-}_{\Z})_{n}[\TRIVREP_\lambda]\subseteq K(\lambda)_\Q\cong U^-_\Q/\sum_{i\in I}U^-_{\Q}f_i^{\lambda(h_i)+1}$
as shown in Theorem \ref{final_thm3}.
By tracing this isomorphism, we see that the graded $\Z$-module isomorphism $K(\infty)^*\cong U^-_\Z$ 
is given by the 
composite 
\begin{align*}
U^{-}_{\Z} \ISOM U^{+}_{\Z}\stackrel{\pi}{\longrightarrow} K(\infty)^*
\end{align*}
where $U^{-}_{\Z} \SISOM U^{+}_{\Z}$ is the algebra anti-isomorphism given by $f_i\mapsto e_i$ for all $i\in I_q$.
See also the proof of ~\cite[Theorem 7.17]{BK} in ~\cite[\S3]{BK2}.
\end{proof}

\end{document}